\documentclass{amsart}
\usepackage[english]{babel}
\usepackage{amsthm,amstext,amsmath,amscd,amssymb,latexsym, mathrsfs}
\usepackage{color}
\usepackage{graphicx}
\usepackage[T1]{fontenc}
\usepackage[variablett]{lmodern}
\usepackage{txfonts}

\usepackage[matrix,arrow]{xy}

\usepackage{amsfonts,enumerate,array,hyperref}

\usepackage{algorithm}
\usepackage{booktabs}
\usepackage{tikz}
\usepackage{caption}
\usepackage{subcaption}
\usetikzlibrary{backgrounds,patterns,matrix,calc,arrows}

\newcommand{\C}{{\mathbb C}}

\newcommand{\N}{{\mathbb N}}
\renewcommand{\P}{{\mathbb P}}

\newcommand{\Ss}{\mathbb{S}}
\newcommand{\T}{{\mathcal T}}

\newcommand{\Z}{{\mathbb Z}}

\newcommand{\kc}{{\mathcal C}}

\newcommand{\ko}{{\mathcal O}}

\newcommand{\kv}{{\mathcal V}}

\newcommand{\s}{\mathscr}

\newcommand{\sI}{{\s I}}

\newcommand{\punkt}{\HHspace{-.3ex}\raise.15ex\HHbox to1ex{\HHuge.}}

\DeclareMathOperator{\rank}{rank}

\newtheorem{theorem}{Theorem}[section]
\newtheorem{lemma}[theorem]{Lemma}
\newtheorem{proposition}[theorem]{Proposition}
\newtheorem{corollary}[theorem]{Corollary}

\theoremstyle{definition}

\theoremstyle{remark}
\newtheorem{remark}[theorem]{Remark}

\numberwithin{equation}{section}

\begin{document}

\title[Most tangential varieties to Veronese varieties are nondefective]
{
Most secant varieties of tangential varieties to Veronese varieties are nondefective
}

\author{Hirotachi Abo}
\address{Department of Mathematics, University of Idaho, Moscow, ID 83844-1103, USA}
\email{abo@uidaho.edu}
\author{Nick Vannieuwenhoven}
\address{Department of Computer Science, KU Leuven, Celestijnenlaan 200A, B-3001 Heverlee, Belgium.}
\email{nick.vannieuwenhoven@cs.kuleuven.be}
\thanks{The first author was partly supported by NSF grant DMS-0901816. The second author's research was partially supported by a Ph.D.~Fellowship of the Research Foundation--Flanders (FWO) and partially by a Postdoctoral Fellowship of the Research Foundation--Flanders (FWO)}
\begin{abstract}
We prove a conjecture stated by Catalisano, Geramita, and Gimigliano in 2002, which claims that the secant varieties of tangential varieties to the $d$th Veronese embedding of the projective $n$-space $\P^n$ have the expected dimension, modulo a few well-known exceptions. It is arguably the first complete result on the dimensions of secant varieties of a classic variety since the work of Alexander and Hirschowitz in 1995. As Bernardi, Catalisano, Gimigliano, and Id\'a demonstrated that the proof of this conjecture may be reduced to the case of cubics, i.e., $d=3$, the main contribution of this work is the resolution of this base case. The proposed proof proceeds by induction on the dimension $n$ of the projective space  via a specialization argument. This reduces the proof to a large number of initial cases for the induction, which were settled using a computer-assisted proof. The individual base cases were computationally challenging problems. Indeed, the largest base case required us to deal with the tangential variety to the third Veronese embedding of $\P^{79}$ in $\P^{88559}$.
\end{abstract}
\maketitle

\section{Introduction}
\label{sec:intro}
Consider the ring of complex polynomials in $(n+1)$ variables, i.e., 
\[R = \C[x_0, x_1, \ldots, x_n] = \bigoplus_{d \geq 0} S_d(\C^{n+1}), \]
where $S_d(\C^{n+1})$ denotes the $d$th symmetric power of $\C^{n+1}$. The homogeneous polynomials of $R$ of degree $d$ will be referred to as $d$-\emph{forms}. Every $d$-form admits a so-called \emph{Waring decomposition} into a sum of powers of linear forms:
\begin{align} \label{eqn_waring_decomposition}
 f = \sum_{i=1}^s (\lambda_{i,0} x_0 + \lambda_{i,1} x_1 + \cdots + \lambda_{i,n} x_n)^d = \sum_{i=1}^s \ell_i^d;
\end{align}
if the length of the decomposition, i.e., $s$, is minimal, then it is called the \emph{Waring rank} (or just \emph{rank}) of $f$. A basic question to ask is: ``What is the rank of the generic $d$-form $f$?'' This question can be rephrased geometrically. 

Let $\P(S_d(\C^{n+1}))$ be the projective space of $S_d(\C^{n+1})$, so that every element of the former space can be written as $[f] = \{\alpha \, f \, | \, \alpha \in \C\setminus\{0\}\}$ for some nonzero $f \in S_d(\C^{n+1})$. Note that a $d$-form  and its nonzero scalar multiple have the same rank. Therefore, the rank of an element $[f] \in \P(S_d(\C^{n+1}))$ can be defined unambiguously as the rank of the $d$-form $f$. 
The set of rank-one elements of $\P(S_d(\C^{n+1}))$ forms a nonsingular, nondegenerate $n$-dimensional subvariety called the \emph{$d$th Veronese variety}, which we denote by $\kv_{n,d}$. An element of $\P(S_d(\C^{n+1}))$ has rank $s$ if and only if it lies in a ``secant $(s-1)$-plane to $\kv_{n,d}$,'' i.e., an $(s-1)$-dimensional linear subspace spanned by~$s$ points of $\kv_{n,s}$. Since the $s$th secant variety $\sigma_s(\kv_{n,d})$ of $\kv_{n,d}$ is defined as the Zariski closure of the union of secant $(s-1)$-planes to $\kv_{n,d}$, it follows that the Zariski closure of the set of rank-$s$ elements of $\P(S_d(\C^{n+1}))$ is precisely $\sigma_s(\kv_{n,d})$.
Hence, the $d$-forms that admit an expression as in (\ref{eqn_waring_decomposition}) of length $s$ constitute a dense constructible subset of the affine cone over $\sigma_s(\kv_{n,d})$.
By definition, we have an ascending chain of varieties 
\[
\kv_{n,d} = \sigma_1(\kv_{n,d}) \subset \sigma_2(\kv_{n,d}) \subset \cdots \subset \sigma_s(\kv_{n,d}) \subset \cdots. 
\]
Since every $d$-form admits a Waring decomposition of finite length, this ascending chain becomes stationary, i.e., we have $\sigma_s(\kv_{n,d}) = \P (S_d(\C^{n+1}))$ for a sufficiently large $s \in \N$. This implies that the problem of determining the rank of the generic $d$-form is equivalent to the problem of finding the least positive integer $s$ such that  $\sigma_s(\kv_{n,d}) = \P (S_d(\C^{n+1}))$.

According to Brambilla and Ottaviani \cite{BO2008}, the origins of \emph{Waring's problem}, i.e., the problem of finding the Waring rank of the generic $d$-form, can be traced back some 150 years to Sylvester, Campbell, Palatini, and Terracini, but it was not until the end of the $20$th century that Alexander and Hirschowitz completely solved the problem in a series of papers culminating in their well-known 1995 paper \cite{AH1995}.

Variations of Waring's problem can be obtained by modifying the concept of the rank of $d$-form. A natural alternative definition of the rank of $d$-form is the \emph{Chow rank}, which was recently considered in \cite{abo2014, AB2011,Shin, Torrance2015}. 
Every $d$-form $f$ in $(n+1)$ variables is expressible as a sum of a finite number of \emph{completely decomposable $d$-forms}, 
i.e., $d$-forms that can be written as products of $d$ linear forms: 
\begin{align}\label{eqn_completely_decomposable}
 f = \sum_{i=1}^s \prod_{j=1}^d (\lambda_{i,j,0} x_0 + \lambda_{i,j,1} x_1 + \cdots + \lambda_{i,j,n} x_n) = \sum_{i=1}^s \ell_{i,1} \ell_{i,2} \cdots \ell_{i,d}.
\end{align}   
The minimum number of completely decomposable $d$-forms that sum up to a $d$-form is then called the \emph{Chow rank} of the $d$-form. We propose to call (\ref{eqn_completely_decomposable}) a \emph{Chow decomposition}. As an analogy to Waring's problem, one may seek the least positive integer $s$ such that a general $d$-form has Chow rank $s$. Geometrically, this \emph{Chow's problem} is equivalent to finding the least positive integer $s$ such that the $s$th secant variety of the \emph{Chow variety} $\kc_{n,d}$ parameterizing the zero cycles of degree $d$ in $\P(\C^{n+1})$ coincides with its ambient space $\P(S_d(\C^{n+1}))$; see, e.g., \cite{GKZ1994} for the definition of the Chow variety.
%
The value of such an $s$ is believed to be $\left\lceil \binom{n+d}{d}/(nd+1) \right\rceil$ for most values of $n$ and $d$. Recently, it was proved that this is actually the case for ternary forms \cite{abo2014}. This is the first nontrivial case, as the conjecture for the binary case is trivially true because of the fundamental theorem of algebra. Nevertheless, in the general case, the conjecture is still wide open, and new ideas are necessary to prove (or disprove) this conjecture. 

In this paper, we attempt to further our understanding of the aforementioned conjecture by investigating an intermediate case between Waring's problem, which was resolved, and Chow's problem, which is wide open.
%
This idea was inspired by the family of problems consisting of decomposing a partially symmetric tensor into a sum of rank-$1$ partially symmetric tensors in $S_{d_1}(\C^{n+1}) \otimes S_{d_2}(\C^{n+1}) \otimes \cdots \otimes S_{d_k}(\C^{n+1})$ for some partition $(d_1, d_2, \ldots, d_k)$ of $d$; see \cite{CGG2004} and \cite[Sections 3.6, 5.5.3, and 5.7]{Landsberg2012} for a geometric interpretation of such decompositions, and see \cite{Abo2010,AB2009,AB2013,LP2013} for recent progress on these problems. This family of problems includes as special cases the original Waring problem, where the partition is $(d)$, as well as the general tensor decomposition problem, where the partition is $(1,1,\ldots,1)$.

We envisage a whole spectrum of decompositions intermediate between the two extreme cases (\ref{eqn_waring_decomposition}) and (\ref{eqn_completely_decomposable}).
For a fixed partition $\boldsymbol d  = (d_1, \dots, d_k)$ of a positive integer $d$, consider a $d$-form that splits completely into linear forms as follows: $\ell_1^{d_1} \ell_2^{d_2} \cdots \ell_k^{d_k}$. Every $d$-form $f \in S_d(\C^{n+1})$ can be expressed as a finite sum of such $d$-forms: 
\begin{equation}
\label{eq:chow-waring}
f = \sum_{i=1}^s \prod_{j=1}^k (\lambda_{i,j,0} x_0 +\cdots + \lambda_{i,j,n} x_n)^{d_j} = \sum_{i=1}^s \prod_{j=1}^k \ell_{i,j}^{d_j}. 
\end{equation}
We call (\ref{eq:chow-waring}) a \emph{$\boldsymbol d$th Chow--Waring decomposition} of $f$. If the decomposition~(\ref{eq:chow-waring}) has the shortest length, then $s$ is called the \emph{$\boldsymbol d$th Chow--Waring rank of $f$}. It goes without saying that (\ref{eq:chow-waring}) is a Waring decomposition of $f$ if $\boldsymbol d = (d)$, while it is a Chow decomposition of $f$ if $\boldsymbol d = (1, \dots, 1)$. Hence, the family of problems of finding the \emph{generic $\boldsymbol d$th Chow--Waring rank} of $\P(S_d(\C^{n+1}))$, i.e., the least positive integer $s$ such that the generic $f \in S_d(\C^{n+1})$ has $\boldsymbol d$th Chow--Waring rank $s$, contains both Waring's and Chow's problems as special cases. 

A geometric interpretation of the above-mentioned problem is as follows. 
Consider the map of $(\P^n)^{k}$ into $\P(S_d(\C^{n+1}))$ defined by $([\ell_1], [\ell_2], \ldots, [\ell_k]) \mapsto [\ell_1^{d_1} \ell_2^{d_2} \cdots \ell_k^{d_k}]$. We call this map the \emph{$\boldsymbol d$th Chow--Veronese map of $\P^n$} and its image, denoted by $CV_{n,\boldsymbol d}$, the \emph{$\boldsymbol d$th Chow--Veronese variety}. The $d$th Veronese embedding $\kv_{n,d}$ of $\P(\C^{n+1})$ and the Chow variety $\kc_{n,d}$ 
in $\P(\C^{n+1})$ are both special types of Chow--Veronese varieties; namely, $CV_{n,(d)} = \kv_{n,d}$ and $CV_{n,(1,\dots, 1)} = \kc_{n,d}$. As in the case of Veronese varieties and Chow varieties, finding the least positive integer $s$ such that the $s$th secant variety of $CV_{n, \boldsymbol d}$ fills $\P(S_d(\C^{n+1}))$ and finding the generic $\boldsymbol d$th Chow--Waring rank of $\P(S_d(\C^{n+1}))$ are equivalent.

The only $\boldsymbol d$th Chow--Waring problem that has been completely settled so far is the case where $\boldsymbol d = (d)$, i.e., the original Waring's problem. This paper shall be concerned with a modest first step towards the resolution of the $\boldsymbol d$th Chow--Waring problem. It will be our goal to determine---in all cases---the $(d-1,1)$th generic Chow--Waring rank of $\P(S_d(\C^{n+1}))$. 

The $(d-1,1)$th Chow--Veronese variety $CV_{n,(d-1,1)}$ is more commonly known as the \emph{tangent variety} $\T_{n,d}$ of $\kv_{n,d}$. Determining the generic $(d-1,1)$th Chow--Waring rank of $\P(S_d(\C^{n+1}))$ is therefore equivalent to finding the least positive integer $s$ such that $\sigma_s(\T_{n,d}) = \P(S_d(\C^{n+1}))$; this is a problem that has already received some attention in the literature \cite{ballico,BCGI,CGG2002}.
As points on $\T_{n,d}$ can be parameterized as $[\ell^{d-1} m]$ with $\ell, m \in \C^{n+1}$, it follows immediately from a standard parameter count that
\begin{eqnarray}
\label{eq:expectedDimension}
 \dim \sigma_s(\T_{n,d}) \leq  \min \left\{ (2n+1)s, \ N(n,d)\right\}-1, \mbox{ where } N(n,d)=\binom{n+d}{d}.
\end{eqnarray}
We call the quantity on the right-hand side of the above inequality the \emph{expected dimension} of $\sigma_s(\T_{n,d})$. The $s$th secant variety $\sigma_s(\T_{n,d})$ to $\T_{n,d}$ is said to be \emph{defective} if equality does not hold in (\ref{eq:expectedDimension}). Otherwise we say that $\sigma_s(\T_{n,d})$ is \emph{nondefective}.
Catalisano, Geramita, and Gimigliano \cite{CGG2002} conjectured in 2002 that $\sigma_s(\T_{n,d})$ is always nondefective, 
unless it is one of the exceptional cases in the statement of Theorem~\ref{thm_main}.
The main contribution of this paper is a proof of the Catalisano--Geramita--Gimigliano (CGG) conjecture. To be more precise, we will prove the following result.
\begin{theorem}\label{thm_main}
The $s$th secant variety $\sigma_s(\T_{n,d})$ of the tangential variety $\T_{n,d}$ to the Veronese variety $\kv_{n,d}$ is nondefective, except in the following cases:
\begin{itemize}
 \item[(i)] $d=2$ and $2 \le 2s < n$; and
 \item[(ii)] $d=3$ and $s = n = 2, 3,$ and $4$.
\end{itemize}
\end{theorem}

To the best of our knowledge, all the defective cases listed in Theorem~\ref{thm_main} were first found by Catalisano, Geramita, and Gimigliano in~\cite{CGG2002}. Please consult~\cite[Propositions 3.2, 3.3, and 3.4]{CGG2002} for the detailed descriptions of these defective cases.

The following is an immediate consequence of Theorem~\ref{thm_main}, which completely solves the $(d-1,1)$th Chow--Waring problem.
\begin{corollary}
The generic $(d-1,1)$th Chow--Waring rank of $\P(S_3(\C^{n+1}))$ is 
\[
 s = \left\lceil \frac{N(n,d)}{2n + 1} \right\rceil,
\]
unless $(d,n,s)$ is one of the defective cases listed in Theorem~\ref{thm_main}. In the defective cases with $d=3$, the generic rank equals $s+1$, while for $d=2$, the generic rank is $1+\lfloor n/2 \rfloor$.
\end{corollary}

We would be remiss if we did not mention that there has been significant previous work toward the completion of the CGG conjecture. In~\cite{CGG2002}, Catalisano, Geramita, and Gimigliano proved that the CGG conjecture is true for the following two cases; namely,  the first is the case where $d=2$ and  $n$, $s$ are arbitrary, and the second is the case where $s\leq 5$ and $n$, $d$ are arbitrary. In particular, the CGG conjecture was proved to be true for $d=3$, $n\in \{2,3,4\}$, and an arbitrary $s$. Another important result was obtained by Ballico. In~\cite{ballico}, he showed that the CGG conjecture holds for $n \in \{2,3\}$ and for arbitrary $d$ and $s$. 

A major breakthrough was made by Bernardi, Catalisano, Gimigliano, and Id\'a in 2009, who showed in \cite[Corollary 2.5]{BCGI} that if the CGG conjecture holds for $d=3$, then it also holds for $d \geq 4$. Additionally, they proved by an explicit computation in the commutative algebra software CoCoA that the CGG conjecture holds for $n \le 9$. Therefore, the novel contribution of this work concerns only the cases where $d=3$ and $n \geq 10$. We prove that there are no defective cases other than (i) and (ii) in Theorem \ref{thm_main}. 

The proof of Theorem~\ref{thm_main} proceeds via an inductive approach based on a specialization technique, which was inspired by the paper of Brambilla and Ottaviani~\cite{BO2008}. This inductive approach reduces the problem to a finite number of base cases. Employing a computer-assisted proof whose key step consists of computing the ranks of several very large integer matrices, we proved that these base cases are true. This turned out to be a computationally challenging problem, which we alleviated by exploiting the particular structure of aforementioned matrices. 

An application in which secant varieties of Chow--Veronese varieties naturally appear is the design of efficient algorithms for evaluating multivariate polynomials. From a practical viewpoint, knowledge of a $\boldsymbol d$th Chow--Waring decomposition of a specific polynomial $f$ results in tremendous savings in the (multiplicative) complexity for the evaluation of $f(x_0, x_1, \ldots, x_n)$. Indeed, if the polynomial is given naively through its coefficients, i.e., $f = \sum_{0 \le i_1 \le \cdots \le i_d \le n} c_{i_1,\ldots,i_d} x_{i_1} \cdots x_{i_d}$, then evaluating it requires $d$ multiplications for each of the $\binom{n+d}{d}$ terms. However, if we know that there is a $k$ such that $f$ admits $(d_1,\ldots,d_k)$th Chow--Waring rank $r$, then we also know that there exists an algorithm for evaluating the polynomial whose multiplicative complexity is only $k(n + 1) + d-1$ multiplications for each of the $r$ terms in the Chow--Waring decomposition, namely by first evaluating each of the linear forms, taking the appropriate products, and then summing. 

Aside from the foregoing very practical application, there is also theoretical interest in complexity theory in finding the minimal multiplicative complexity of simple computations involving only multiplication and summation. They are modeled as \emph{arithmetic circuits} which are essentially finite, labeled, directed, acyclic graphs; see, e.g., \cite{Landsberg2015,Burgisser1997,Shpilka2010}. Of particular interest are so-called depth-$3$ $\Sigma\Pi\Sigma$ circuits which are trees with $3$ levels representing precisely the polynomials that can be written as in (\ref{eqn_completely_decomposable}), where $s$ corresponds precisely with the number of inbound edges at the root of the tree \cite{Shpilka2010}. The overarching goal in arithmetic complexity theory consists of finding low-degree families of explicit polynomials that nevertheless admit superpolynomial growth of their arithmetic circuit size because of their ramifications to the separation of various algebraic complexity classes \cite{Shpilka2010,Burgisser2000}. Since several notions of the size of the circuit all depend intrinsically on the rank $s$ of the Chow decomposition, it is essentially the study of the secant varieties of Chow varieties. Similarly, the honest part of the $s$-secant variety of a $\boldsymbol d$th Chow--Veronese variety corresponds with a special subset of depth-$3$ $\Sigma\Pi\Sigma$ circuits that could be easier to study---as evidenced by the fact that we know the dimensions of these varieties already in two cases: $\boldsymbol d = (d)$ by the Alexander--Hirschowitz theorem, and $\boldsymbol d = (d-1,1)$ by Theorem \ref{thm_main}.

The outline of the remainder of this paper is as follows. In the next section, we recall some basic properties of tangential varieties of Veronese varieties. Section~\ref{sec:induction} describes the aforementioned inductive approach and illustrates how it will be used to prove Theorem~\ref{thm_main}. The main goal of this section is to state the inductive step (Proposition~\ref{prop:induction}) as well as list the base cases of the induction (Corollary~\ref{cor_base_cases}). Section~\ref{sec_inductive} will be devoted to the proof of Proposition~\ref{prop:induction}. In Section~\ref{sec_base}, we will verify the base cases to complete the proof of Theorem~\ref{thm_main}.

\paragraph{\bf Acknowledgements} We would like to thank G.~Ottaviani for providing feedback on an earlier version of this manuscript.

\section{Preliminaries}
\label{sec:prelim}
This section recalls some basic results on secant varieties of tangential varieties to Veronese varieties. First, we introduce some notation. Let $U$ be an $(n+1)$-dimensional vector space over $\C$. We denote by $\P(U)$, or simply $\P^n$, the projective space of lines in $U$ passing through the origin. Throughout this paper, we write $[v]$ for the equivalence class containing a nonzero vector $v$ of $U$.  For any closed subscheme $X$ of $\P^n$, we denote by 
\begin{itemize}
\item[--] $I_X$ the ideal of $X$; 
\item[--] $\sI_X$ the ideal sheaf of $X$; 
\item[--] $h_{\P(U)}(X,-)$ the Hilbert function of $X$; and 
\item[--] $\widehat{X} \subseteq U$ the affine cone over $X$. 
\end{itemize}
If $X \subset \P(U)$ is a variety, then $T_p(X)$ denotes the projective tangent space to $X$ at $p \in X$.

\subsection{Tangential varieties}
\label{sec:tangential-varieties}
Let $\widetilde{\nu}_{n,d} : U \rightarrow S_d(U)$ be the map defined by sending $\ell \in U$ to $\ell^d \in S_d(U)$. This map induces the so-called \emph{$d$th Veronese map} $\nu_{n,d}:\P(U) \rightarrow \P(S_d(U))$. The image of $\nu_{n,d}$ is the Veronese variety $\kv_{n,d}$.
Let $d\widetilde{\nu}_{n,d}: T_{\ell}(U) \rightarrow T_{\ell^d}\bigl(S_d(U)\bigr)$ be the differential of $\widetilde{\nu}_{n,d}$ at $\ell \in U \setminus \{0\}$, where $T_{\ell}(U)$ is the tangent space to $U$ at $\ell$ and $T_{\ell^d}\bigl(S_d(U)\bigr)$ is the tangent space to $S_d(U)$ at $\ell^d$. By taking the derivative of the parametric curve $(\ell+tm)^d$, $m \in U$, one can show that $\ell^{d-1}U$ is the image of $d\widetilde{\nu}_{n,d}$, i.e., the affine cone $\widehat{T}_{[\ell^d]} (\kv_{n,d})$ over the tangent space $T_{[\ell^d]} (\kv_{n,d})$ to $\kv_{n,d}$ at $[\ell^d]$. 

Let $\T_{n,d}$ denote the tangential variety of $\kv_{n,d}$, i.e., 
\[
 {\T}_{n,d} = \overline{ \bigcup_{[\ell^d] \in \kv_{n,d}} T_{[\ell^d]} (\kv_{n,d}) },
\]
where the overline denotes the Zariski closure in $\P(S_d(U))$. Let $\widehat{\T}_{n,d}$ be the affine cone over $\T_{n,d}$ in $S_d(U)$. Define a map $\varphi: U \times U \rightarrow S_d(U)$ by $\varphi(\ell, m) = \ell^{d-1} m$.  Again, taking the derivative of the parametric curve $(\ell+t\ell')^{d-1}(m+tm')$ with $\ell', m' \in U$ proves that the image of the differential $d \varphi$ of $\varphi$ at a generic point $(\ell, m) \in U \times U$ is given by 
\begin{align}
\label{eq:tangentSpace}
\widehat{T}_{[\ell^{d-1}m]} (\T_{n,d}) = \ell^{d-1}U+\ell^{d-2}mU.
\end{align}
Hence, it follows that $\T_{n,d}$ is irreducible of dimension $2n$ and that its singular locus is $\kv_{n,d}$; see Proposition 1.1 in \cite{CGG2002} for more details.

The $s$th secant variety $\sigma_s(\T_{n,d})$ of the tangential variety $\T_{n,d}$ is defined as the Zariski closure of the projectivization of the image of the map
\begin{align*}
 \widetilde{\sigma}_{s}: 
\left(\widehat{\T}_{n,d}\right)^s = \widehat{\T}_{n,d} \times \cdots \times  
\widehat{\T}_{n,d} &\rightarrow S_d(\C^{n+1}) \\
(\ell_1^{d-1}m_1, \dots, \ell_s^{d-1}m_s) &\mapsto \sum_{i=1}^s \ell_i^{d-1}m_i.
\end{align*}
Its dimension satisfies inequality (\ref{eq:expectedDimension}). If equality holds, or, equivalently, if the $s$th secant variety of the tangential variety $\T_{n,d}$ is nondefective, then we say that \emph{the statement $T(n, d; s)$ is true}, otherwise we say that it is \emph{false}. If $(2n+1)s \leq N(n,d)$, then we say that $T(n, d; s)$ is \emph{subabundant}. On the other hand, if $(2n+1)s \geq N(n,d)$, then we say that $T(n, d; s)$ is \emph{superabundant}. A statement that is simultaneously subabundant and superabundant is called \emph{equiabundant}.

For determining the dimension of $\sigma_s(\T_{n,d})$, we may consider the dimension of the projective  tangent space $T_{p}\bigl( \sigma_s(\T_{n,d}) \bigr)$ at a generic point $p \in \sigma_s( \T_{n,d} )$. The following classic theorem, which is essentially due to Terracini \cite{Terracini1911}, describes the tangent space to the secant variety $\sigma_s(\T_{n,d})$ at a generic point in terms of tangent spaces to $\T_{n,d}$.

\begin{lemma}[Terracini's lemma] 
\label{th:terracini}
Let $[\ell_1^{d-1}m_1], \dots, [\ell_s^{d-1}m_s] \in \T_{n,d}$ be $s$ generic points, and let $q$ be a generic point of the subspace 
of $\P(S_d(U))$ spanned by them. 
Then, the affine cone over the tangent space at $q$ is given by
\[
\widehat{T}_q \bigl( \sigma_s (\T_{n,d} ) \bigr) = \sum_{i=1}^s \widehat{T}_{[\ell_i^{d-1}m_i]}(\T_{n,d}).
\] 
\end{lemma}
\begin{proof}
The tangent space to $(\widehat{\T}_{n,d})^s$ at $(\ell_1^{d-1}m_1, \dots, \ell_s^{d-1}m_s)$ is given by $\bigoplus_{i=1}^s \widehat{T}_{[\ell_i^{d-1} m_i]}(\T_{n,d})$. Since $\widetilde{\sigma}_s$ is obtained from 
the map $(S_d(U))^s \rightarrow S_d(U)$ defined by sending $(v_1, \dots, v_s)$ to $\sum_{i=1}^s v_i$ by restricting to $(\widehat{\T}_{n,d})^s$, the image of the differential of $\widetilde{\sigma}_s$ at $\sum_{i=1}^s \ell_i^{d-1}m_i$ is $\sum_{i=1}^s \widehat{T}_{[\ell_i^{d-1} m_i]}\left(\T_{n,d}\right)$, which completes the proof. 
\end{proof}
\begin{remark}
\label{rem:semi-continuity}
Terracini's lemma will be the main computational tool that we employ in Section~\ref{sec_base} for proving the truths of the base cases of our inductive proof of Theorem~\ref{thm_main}. Note that if $\sum_{i=1}^s \widehat{T}_{p_i}(\T_{n,d})$ has the expected dimension for some specific choice of points $p_1, \ldots, p_s \in \T_{n,d}$, then the foregoing linear space also has the expected dimension for a generic choice of points, by semicontinuity. Thus, in order to prove the truth of $T(n,d;s)$, it suffices to show that $\dim \sum_{i=1}^s \widehat{T}_{p_i}(\T_{n,d})$ equals the expected value for any convenient choice of points $p_1, \ldots, p_s \in \T_{n,d}$.
\end{remark}
\subsection{Hilbert function} \label{sec_hilbert_function}
Let $p \in \T_{n,d}$ be generic. Then, there exist two distinct points $[\ell], [m] \in \P(U)$ such that $p = [\ell^{d-1}m]$. Let $L$ be the line in $\P(U)$ passing through $[\ell]$ and $[m]$. We call the zero-dimensional closed subscheme of $\P(U)$ defined by $I_{[\ell]}^3+I_L^2$ the \emph{$(2,3)$-point} associated with $p$ and denote it by $p^{2,3}$. 

In~\cite[Section~2]{CGG2002}, Catalisano, Geramita, and Gimgliano demonstrated that the subspace $H^0 \bigl(\P(U), \sI_{p^{2,3}}(d)\bigr)$ formed by hypersurfaces of degree $d$ in $\P(U)$ containing $p^{2,3}$ can be identified with the subspace of hyperplanes in $\P(S_d(U))$ containing $T_p(\T_{n,d})$. In particular, the Hilbert function is
\begin{align*}
h_{\P(U)}\Bigl(p^{2,3},d\Bigr) = N(n,d) - \dim \left(I_{[\ell]}^3+I_L^2\right)_d 
&=   N(n,d)  - \dim H^0 \left(\P(U), \sI_{p^{2,3}}(d)\right) \\
&=  \dim \widehat{T}_p(\T_{n,d}) \\
&=  2n+1.  
\end{align*}

Let $p_1, \ldots, p_s \in \T_{n,d}$ be generic points, let $p_1^{2,3}, \ldots, p_s^{2,3}$ be their respective associated~$(2,3)$-points, and let $Z = \{p_1^{2,3}, \ldots, p_s^{2,3}\}$. Since 
\[
H^0\left(\P(U), \sI_{Z}(d)\right) = (I_Z)_d = \left(\bigcap_{i=1}^s I_{p_i^{2,3}}\right)_d = \bigcap_{i=1}^s \left(I_{p_i^{2,3}}\right)_d   = \bigcap_{i=1}^s H^0\left(\P(U), \sI_{p_i^{2,3}}(d)\right),  
\]
we can view $H^0\left(\P(U), \sI_{Z}(d)\right)$ as the subspace of $S_d(U)$ spanned by hyperplanes containing the linear span of $T_{p_1}(\T_{n,d}), \dots, T_{p_s}(\T_{n,d})$. Furthermore, since each $p_i^{2,3}$ imposes $2n+1$ linearly independent 
conditions on hypersurfaces of degree $d$, 
we have 
\[
\dim H^0\left(\P(U), \sI_{Z}(d)\right) \geq \max \left\{ N(n,d) - (2n+1)s, \, 0\right\},  
\]
or, equivalently, 
\begin{equation}
\label{eq:hilbert-function}
h_{\P(U)}(Z,d)  =  N(n,d) - \dim H^0\left(\P(U), \sI_{Z}(d)\right) \leq \min\left\{(2n+1)s, \, N(n,d)\right\},  
\end{equation}
with equality occurring if and only if one of the following holds: 
\begin{itemize}
\item[(i)] $s \leq \left\lfloor N(n,d) /(2n+1)\right\rfloor$ and all the $p_i^{2,3}$ impose linearly independent conditions on hypersurfaces of degree $d$; 
\item[(ii)] $s \geq \left\lceil N(n,d) /(2n+1)\right\rceil$ and there are no hypersurfaces of degree $d$ containing $Z$. 
\end{itemize}
Therefore, it follows from Lemma~\ref{th:terracini} that showing the truth of $T(n,d;s)$ is equivalent to showing that the Hilbert function of $Z$ has the expected value at $d$. 

\section{Induction}
\label{sec:induction}
In the remainder, we shall be concerned with proving Theorem \ref{thm_main} in the case $d=3$, as the general case $d\ge4$ then follows from \cite[Corollary 2.5]{BCGI}. For brevity, we shall write $\kv_n$ for the third Veronese embedding of $U$ in $\P(S_3(U))$ and $\T_n$ for its tangential variety. The affine cone over $\T_n$ will be denoted by $\widehat{\T}_n$ and the dimension of $\P(S_3(U))$ is denoted by $N(n) = N(n,3) = \binom{n+3}{3}$. The main purpose of this section is to introduce the inductive method that we pursue for proving that the secant varieties $\sigma_s(\T_n)$ of $\T_n$ have the expected dimension except for $s=n=2,3,4$.

\subsection{Subabundance and superabundance}\label{sec_sub_sup}
Proving nondefectivity of all $s$th secant varieties of $\T_{n}$ is simplified by the fact that usually only two cases need to be proved, namely $\sigma_{s_1}( \T_{n,d} )$ and $\sigma_{s_2}( \T_{n,d} )$, where $s_1$ is the largest integer such that $T(n, 3; s_1)$ is subabundant and $s_2$ is the least integer such that $T(n, 3; s_2)$ is superabundant. Naturally $s_2 - s_1 \le 1$. We claim that $s_1$ and $s_2$ have the following explicit expressions if $n \ge 8$. 
For such an $n$, let $k$ and $r$ be the quotient and the remainder after division of $n$ by $24$. Define
\[
 s_1(n) = 48k^2+(11+4r)k+\left\lfloor (4r^2+22r+33)/48 \right\rfloor \;\mbox{ and }\; 
 s_2(n) = s_1(n) + 1.
\]
Note that, for $i \in \{1,2\}$ and $n \ge 32$, we may define
\[
 t(n) = s_i(n) - s_i(n-24) = 96k + 4r - 37 = 4n-37,
\]
and, hence, for $n \ge 56$ one has $t(n) - t(n-24) = 96$. The next lemma entails that $T(n,3;s_1(n))$ is subabundant and $T(n,3;s_2(n))$ is superabundant.

\begin{proposition}
\label{eq:s}
Let $n \geq 8$. Then, $s_1(n) = \left\lfloor N(n)/(2n+1) \right\rfloor$ and $s_2(n) = \left\lceil N(n)/(2n+1) \right\rceil$. 
\end{proposition}
\begin{proof}
%
As before, we write $k$ and $r$ for the quotient and remainder in the division of $n$ by $24$ respectively. Consider~$N(24k+r)$ as a polynomial in $k$. Then the quotient and remainder when dividing it by $48k+2r+1 = 2(24k+r)+1$ are  
$48k^2+(4r+11)k+(4r^2+22r+33)/48$ and $5/16$ respectively. Let $M(24k+r) = N(24k+r)- \bigl(48k^2+(4r+11)k\bigr)(48k+2r+1)$, and let $f(r) = (4r^2+22r+33)/48$. Then,
\[
\frac{M(24k+r)}{48k+2r+1} = f(r)+\frac{5}{16(48k+2r+1)}. 
\]
Therefore, in order to prove Proposition~\ref{eq:s}, it suffices to show that 
\[
\left\lfloor \frac{M(24k+r)}{48k+2r+1} \right\rfloor = \left\lfloor  f(r) \right\rfloor 
\]
for $n \ge 8$. To do so, it is enough to show that 
\begin{equation}
\label{eq:integer-part}
0 <\; f(r)-\left\lfloor f(r) \right\rfloor + \frac{5}{16(48k+2r+1)} < 1.
\end{equation}

Assume that $k = 0$, i.e., $n=r$. Then, straightforward calculations show that 
\[
0 <\; f(r)-\left\lfloor f(r) \right\rfloor + \frac{5}{16(2r+1)} \leq 1
\]  
with equality if and only if $r\in \{1,2,7\}$. This also implies that inequality~(\ref{eq:integer-part}) holds for every $k \in \N$ and for every $r \in \{0, \dots, 23\}$, because 
\[
0 <\; \frac{5}{16(48k+2r+1)} < \frac{5}{16(2r+1)} 
\]
for such a $k$. We can therefore conclude that if  $n\in \N$ with $n \geq 8$, then $N(n)/(2n+1) \not\in \N$ and $s_1(n) = \left\lfloor N(n)/(2n+1) \right\rfloor$, from which it follows immediately that~$s_2(n) = \left\lceil N(n)/(2n+1) \right\rceil$. Thus, we completed the proof. 
\end{proof}
The foregoing implies that $s_1(n)$ is the largest value of $s$ for which $T(n, 3; s)$ is subabundant, and, similarly, $s_2(n)$ is the smallest value of $s$ for which $T(n, 3; s)$ is superabundant. The following is an immediate consequence of Lemma~\ref{th:terracini}. 
\begin{lemma}
Let $s \ge 2$. We have:
\begin{itemize}
 \item[(i)] If $T(n, d; s)$ is true and subabundant, then $T(n, d; s-1)$ is also true and subabundant; and
 \item[(ii)] If $T(n, d; s)$ is true and superabundant, then $T(n, d; s+1)$ is also true and superabundant.
\end{itemize}
\end{lemma}
Consequently, for concluding the proof of Theorem~\ref{thm_main}, it suffices to demonstrate that both $T(n, 3; s_1(n))$ and $T(n, 3; s_2(n))$ are true. It is important to recall that the requirement $n \ge 8$ in the explicit definitions of $s_1(n)$ and $s_2(n)$ is not a limitation because $T(n,d;s)$ is known to be true for $n \le 9$. It is also interesting to note that $T(2, 3; 2)$ and $T(7, 3; 8)$ are the only statements that are equiabundant; the former statement is false and the latter statement will be shown to be true. The fact that $T(7, 3; 8)$ is equiabundant is the reason why we restrain ourselves to $n\ge8$ in our proof of Theorem \ref{thm_main}, as it would introduce some unpleasant special cases in the inductive proof that will be considered in Section \ref{sec_specialization}.

\subsection{Convenient subspaces} \label{sec_convenient_subspaces}
In order to prove that $T(n,3;s_i(n))$ is true for each $i \in \{1,2\}$, we specialize a subset of the $s_i(n)$ points of $\T_n$ to $24$-codimensional linear sections of $\T_n$ and show that the linear span of the tangent spaces to $\T_n$ at such points has the expected dimension.

Let $B = \{x_0, \dots, x_n\}$ be a basis of $U$, and then we define the standard subspaces as $U_i = \mathrm{Span}(B \setminus B_i)$ with $B_i = \{x_{24(i-1)}, \dots, x_{24i-1}\}$ for $i=1,2,$ and $3$. We also introduce the notation $\overline{U}_i = \mathrm{Span}(B_i)$ and $L_i = \P(U_i)$. Finally, we shall denote the third Veronese embedding of $U_i$ by 
\(
 \Ss_i = \P(S_3(U_i)) \mbox{ for } i \in \{1,2,3\},
\)
and $\Ss$ denotes $\P(S_3(U))$.

In the next sections, we regularly exploit the following basic property of the subspaces $U_i$: Let $K \subseteq \{1,2,3\}$. Then, 
\begin{align*}
\bigcap_{j \in K} \Ss_j = \P\Bigl( S_3\bigl( \cap_{j \in K} U_j \bigr) \Bigr).
\end{align*}
A proof of the above statement is obtained by observing that $B \setminus \bigl(\cup_{j\in K} B_j\bigr)$ is a basis for the affine cone over the left-hand side of the equality.
As a consequence, we also find
\begin{align*}
p = [\ell^{d-1} m] \in \T_{n,d} \cap \bigl(\cap_{j \in K} \Ss_j\bigr) \;\mbox{ iff }\; \ell, m \in \bigcap_{j \in K} U_j.
\end{align*}
These statements can be proved in a straightforward manner.

We start with the following observation. Consider any $i \in\{1,2,3\}$, and let $q \in \T_n \cap \Ss_i$ be generic. Then, it follows from the above observation that there exist generic $\ell, m \in U_i$ such that $q = [\ell^2m]$. The affine cone over the tangent space at $q$ to $\T_n$ is given by $\widehat{T}_q(\T_n) = \ell^2 U + \ell m U$ (see~(\ref{eq:tangentSpace})). Note that adding $S_3(U_i)$ to the previous equality results in $\widehat{T}_q(\T_n) + S_3(U_i) = \ell^2\overline{U}_i + \ell m \overline{U}_i + S_3(U_i)$. This sum of vector spaces is direct when $\ell \not\in [m]$, which is satisfied by a generic choice of $\ell, m \in U_i$ if $\dim U_i = n - 24 > 1$. 
It follows, therefore, that 
\begin{equation}
\label{eq:vector-space-sum}
\dim \left( S_3(U_i) + \widehat{T}_q(\T_n)\right) = N(n-24)+48.
\end{equation}

The next observation concerns the dimension of the vector space sum of some of the $S_3(U_i)$'s. 
\begin{lemma} \label{lem_basis}
Let $1 \le k \le 3$. Then we have
\[
 \dim \sum_{j=1}^k S_3(U_j) = \sum_{j=1}^k (-1)^{j-1} \binom{k}{j} N(n-24j).
\]
\end{lemma}
\begin{proof}
With our choice of basis of $U$, a basis of $S_3(U)$ is given by $E = \{ x_i x_j x_k \;|\; 0 \le i \le j \le k \le n\}$. Similarly, given our choice of the subspaces $U_l$, it follows that a basis of $S_3(U_l)$, $l=1,2,3$, is given explicitly by 
\[
F_l = \Bigl\{ x_i x_j x_k \;|\; 0 \le i \le j \le k \le n,\; i,j,k \not\in \{ 24(l-1), 24(l-1)+1, \ldots, 24l-1 \} \Bigr\}.
\]
Joining the bases, we note that determining the dimension of the span of $F = F_1 \cup \cdots \cup F_k$ becomes a particularly simple task because any two elements $v, w \in F$ are either equal or distinct basis vectors of $E$. Hence, it suffices to count the number of distinct vectors $x_i x_j x_k$ appearing in $F$. This is the problem of computing $|F_1 \cup \cdots \cup F_k|$, which can be solved directly with the Inclusion--Exclusion Principle, concluding the proof.
\end{proof}

Based on (\ref{eq:vector-space-sum}) and Lemma~\ref{lem_basis}, we are led to the following proposition.
\begin{proposition}
\label{prop_statement}
Let $\chi:\N \cup \{0\} \rightarrow \{0,1\}$ be the function  defined by 
\[
 \chi (a) = \left\{ 
 \begin{array}{ll}
 0 & \mbox{if $a = 0$}, \\
 1 & \mbox{if $a \not =0$}. 
 \end{array}
 \right. 
\]
For a given nonnegative integer $s$, let $p_1, \dots, p_s$ be generic points of $\T_n$. For each $i \in \{1,2,3\}$ and a triplet $(a_1, a_2, a_3)$ of nonnegative integers, let $q_{i,1}, \dots, q_{i,a_i}$ be generic points of $\T_n \cap \Ss_i$. Let $k = \chi(a_1)+\chi(a_2)+\chi(a_3)$. Then, the subspace 
\begin{equation}
\label{eq:subspace} 
 W(n,s;a_1,a_2,a_3) = \sum_{i=1}^3 \chi(a_i)\left( S_3(U_i) + \sum_{j=1}^{a_i} \widehat{T}_{q_{i,j}} \left( \T_n \right) \right) + 
 \sum_{i=1}^s \widehat{T}_{p_i}\left(\T_n \right)
\end{equation}
of $S_3(U)$ has at most dimension
\begin{equation}
 \label{eq:dimension}
\min \left\{ 
\sum_{j=1}^k (-1)^{j-1} \binom{k}{j} N(n-24j)+48\sum_{i=1}^3 a_i+ (2n+1)s, \ N(n) \right\}. 
\end{equation}
\end{proposition}

Following \cite{AOP2009}, we introduce the following terminology. Let $w(n,s;a_1,a_2,a_3)$ be the integer~(\ref{eq:dimension}) and let $T(n, s;a_1,a_2,a_3)$ be the statement that (\ref{eq:subspace}) has the expected dimension, i.e.,  it has dimension $w(n,s;a_1,a_2,a_3)$. Note that  $T(n,s;0,0,0)$ is the same as $T(n, 3;s)$. 
If $w(n, s; a_1, a_2, a_3)$ equals the first item in the minimization in (\ref{eq:dimension}), then we say that the statement $T(n, s; a_1, a_2, a_3)$ is \emph{subabundant}. On the other hand, if $w(n, s; a_1, a_2, a_3) = N(n)$, then we say that the statement $T(n, s; a_1, a_2, a_3)$ {is} \emph{superabundant}. If $T(n, s; a_1, a_2, a_3)$ is simultaneously subabundant and superabundant, then we say that it is \emph{equiabundant}.

\begin{remark}
\label{rem:hilbert-function-2}
Recall from Section~\ref{sec_hilbert_function} that if $p \in \T_n$ is generic, then $p^{2,3}$ denotes the $(2,3)$-point associated with $p$. 
Let $p_1, \dots, p_s \in \T_n \setminus \bigcup_{i=1}^3 \Ss_i$. For each $i \in \{1,2,3\}$, let $q_{i,1}, \dots, q_{i,a_i} \in \T_n \cap \Ss_i$. For each $i \in \{1,2,3\}$, the symbol $\chi(a_i)L_i$ means $L_i$ if $a_i \not=0$ and $\emptyset$ otherwise. Let $L = \bigcup_{i=1}^3 \chi(a_i)L_i$. Consider the following zero-dimensional subscheme of $\P(U)$: 
\[
Z = \bigcup_{i=1}^3 \left\{ q_{i,1}^{2,3}, \dots, q_{i,a_i}^{2,3}\right\} \cup  \left\{p_1^{2,3}, \dots, p_s^{2,3}\right\}. 
\]
As a cubic hypersurface contains $L$ if and only if the corresponding hyperplane in $\Ss$ contains the third Veronese embedding of $L$ in $\Ss$, it follows that $H^0(\P(U),\sI_{Z \cup L}(3))$ can also be identified with the subspace spanned by the hyperplanes containing the third Veronese embedding of $L$ as well as the linear span of the tangent spaces at the corresponding points.
Therefore, the subspace $H^0\left(\P(U), \sI_{Z \cup L}(3)\right)$ spanned by the cubic hypersurfaces of $\P(U)$ containing $Z \cup L$, or, equivalently, the linear span of $Z \cup L$, can be identified with the subspace of $\Ss$ spanned by the hyperplanes containing the projectivization of $W(n,s;a_1,a_2,a_3)$. 
In particular, the value of the Hilbert function of $Z\cup L$ at $3$, i.e.,  
\begin{align*}
h_{\P(U)}\left(Z\cup L,3\right) = N(n) - \dim H^0\left(\P(U), \sI_{Z\cup L}(3)\right) \leq N(n) - w(n,s;a_1,a_2,a_3), 
\end{align*}
equals the dimension of $W(n,s;a_1,a_2,a_3)$. This implies that the following are equivalent: 
\begin{itemize}
 \item[(i)] The statement $T(n,s;a_1,a_2,a_3)$ is true.
 \item[(ii)] The value of  the Hilbert function of $Z\cup L$ at $3$ is equal to $w(n,s;a_1,a_2,a_3)$. 
 \item[(iii)] The dimension of $H^0\left(\P(U), \sI_{Z\cup L}(3)\right)$ is equal to 
 \(
 N(n)-w(n,s;a_1,a_2,a_3).
 \)
\end{itemize}
Our approach for proving the truth of statement $T(n, s; a_1, a_2, a_3)$ consists of showing (iii).
\end{remark}

\subsection{The specialization} \label{sec_specialization}
The proof of Theorem \ref{thm_main} proceeds, very roughly speaking, by induction on $n$ via a specialization technique, which follows closely Brambilla and Ottaviani's idea \cite{BO2008} for proving that the secant varieties of the third Veronese variety are (mostly) nondefective. Proposition \ref{prop:induction} demonstrates what needs to be done to show the initial step. The base cases of the induction are listed in Corollary \ref{cor_base_cases} and will be proved with the aid of a computer in Section \ref{sec_base}. The purpose of this subsection is to outline and visualize the main idea of the proof of Proposition \ref{prop:induction}. We give a full proof of this proposition in Section \ref{sec_inductive}, which is based on Remark \ref{rem:hilbert-function-2}. The reader who is familiar with the language of schemes may safely skip to Section \ref{sec_inductive}. 

The basic idea of the aforementioned approach consists of reducing the problem of proving the truth of the statement $T(n,3;s_i(n))$ to the problem of proving $T(n-24,3;s_i(n-24))$, where $i \in \{1,2\}$. For establishing such an inductive approach, we select $s_i(n)$ points~$p_1$, $\ldots$, $p_{s_i(n)}$ of $\T_n$ as follows:
\begin{itemize}
\item[(i)] $96$ generic points on $\T_n \cap  \Ss_j$ for each $j \in \{1,2,3\}$; 
\item[(ii)] $t(n-48)$ generic points on $\T_n \cap \Ss_j \cap \Ss_k$ for each $j ,k \in \{1,2,3\}$ with $j <k$; and 
\item[(iii)] $s_i(n-72)$ points on $\T_n \cap \bigl(\bigcap_{j=1}^3  \Ss_j\bigr)$.
\end{itemize}
As can be easily verified, $s_i(n-24)$ points are placed in $\T_{n-24} = \T_n \cap \Ss_j$ for each $j \in \{1,2,3\}$. The goal is to show that the linear span of the tangent spaces to $\T_n$ at these specific points has the expected dimension for each $n$, i.e., 
\begin{equation}
\label{eq:equality}
\dim \sum_{i=1}^{s_i(n)} \widehat{T}_{p_i} \left(\T_n\right) = w(n,s_i(n);0,0,0),
\end{equation}
under the assumption that $T(n-24,3;s_i(n-24))$ is true. 
Then, by Remark~\ref{rem:semi-continuity}, we could conclude that the statement $T(n,3;s_i(n))$ is true. 

The strategy to prove equality~(\ref{eq:equality}) under the induction hypothesis is to divide the process of specializing the $s_i(n)$ points as indicated in (i)-(iii) into four steps, the first three of which are visualized in Figure~\ref{fig_general_proof_strategy}. In these steps, we have to handle statements $T$ for more general quintuples.

\begin{figure}
\begin{minipage}[b]{0.65\textwidth} \centering
\scalebox{0.85}{
\begin{tikzpicture}[scale=0.9]
\begin{scope}[shift={(0.0,-.45)},rotate=60]
\node at (6.693298e-01,4.106863e+00) {\color{red}\textbullet};
\node at (1.163549e+00,3.174364e+00) {\color{red}\textbullet};
\node at (1.656918e+00,3.648572e+00) {\color{red}\textbullet};
\node at (5.544232e-01,4.591433e+00) {\color{red}\textbullet};
\node at (1.464178e+00,4.026018e+00) {\color{red}\textbullet};
\node at (1.611964e+00,4.231395e+00) {\color{red}\textbullet};
\node at (1.911964e+00,3.831395e+00) {\color{red}\textbullet};
\node at (0.811964e+00,3.631395e+00) {\color{red}\textbullet};
\node at (1.211964e+00,3.431395e+00) {\color{red}\textbullet};
\node at (2.211964e+00,3.131395e+00) {\color{red}\textbullet};
\node at (0.65,2.4) {\color{red} $\star$};
\node at (0.45,1.35) {\color{red} $\star$};
\node at (1.25,1.55) {\color{red} $\star$};
\node at (1.05,1.82) {\color{red} $\star$};
\node at (1.55,2.5) {\color{red} $\star$};
\end{scope}
\begin{scope}[shift={(-0.8,2.16)},rotate=-60]
 \draw[thick,black] (0,0) rectangle (3,5);
\node at (6.693298e-01,4.106863e+00) {\color{blue}\textbullet};
\node at (1.163549e+00,3.174364e+00) {\color{blue}\textbullet};
\node at (1.656918e+00,3.648572e+00) {\color{blue}\textbullet};
\node at (5.544232e-01,4.591433e+00) {\color{blue}\textbullet};
\node at (1.464178e+00,4.026018e+00) {\color{blue}\textbullet};
\node at (1.611964e+00,4.231395e+00) {\color{blue}\textbullet};
\node at (0.611964e+00,3.031395e+00) {\color{blue}\textbullet};
\node at (0.911964e+00,3.431395e+00) {\color{blue}\textbullet};
\node at (1.011964e+00,3.731395e+00) {\color{blue}\textbullet};
\node at (1.111964e+00,4.331395e+00) {\color{blue}\textbullet};
\node at (1.45,2.4) {\color{blue} $\star$};
\node at (1.15,2.4) {\color{blue} $\star$};
\node at (1.65,2.7) {\color{blue} $\star$};
\node at (1.25,2.8) {\color{blue} $\star$};
\node at (1.75,2.5) {\color{blue} $\star$};
\end{scope}
\node at (0.85,1.2) {\color{blue} $\blacklozenge$};
\node at (0.65,0.90) {\color{blue} $\blacklozenge$};
\node at (0.90,0.7) {\color{blue} $\blacklozenge$};
\node at (0.50, 1.8) {\color{blue} $\star$};
\node at (0.91, 1.59) {\color{blue} $\star$};
\node at (0.82, 1.85) {\color{blue} $\star$};
\node at (1.02, 1.82) {\color{blue} $\star$};
\node at (1.42, 1.69) {\color{blue} $\star$};
\node at (-1.294999e+00,43.151443e-01) {\color{violet}\textbullet};
\node at (-1.177258e+00,39.203261e-01) {\color{violet}\textbullet};
\node at (-0.416572e+00,39.266308e-01) {\color{violet}\textbullet};
\node at (-1.323341e+00,36.106281e-01) {\color{violet}\textbullet};
\node at (-0.117399e+00,40.018191e-01) {\color{violet}\textbullet};
\node at (-1.042950e+00,35.370621e-01) {\color{violet}\textbullet};
\node at (-0.842950e+00,42.370621e-01) {\color{violet}\textbullet};
\node at (-1.442950e+00,36.370621e-01) {\color{violet}\textbullet};
\node at (-0.142950e+00,44.370621e-01) {\color{violet}\textbullet};
\node at (-1.942950e+00,36.870621e-01) {\color{violet}\textbullet}; 
\node at (-1.5,2.6) {$t(n)$};
\node at (4,2.3) {\color{black} $s_i(n-24)$};
\node at (4.70,1.35) {\color{black} $\Ss_1$};
\end{tikzpicture}
}
\subcaption{$T(n, t(n); s_i(n-24), 0, 0)$ and ${\color{black} T(n-24, s_i(n-24); 0, 0, 0)}$}
\label{fig_step2}
\end{minipage}\vspace{1em}

\begin{minipage}[b]{0.9\textwidth} \centering
\scalebox{0.85}{
\begin{tikzpicture}[scale=0.9]
\begin{scope}[shift={(0.0,-.45)},rotate=60]
 \draw[thick,black] (0,0) rectangle (3,5);
\node at (6.693298e-01,4.106863e+00) {\color{red}\textbullet};
\node at (1.163549e+00,3.174364e+00) {\color{red}\textbullet};
\node at (1.656918e+00,3.648572e+00) {\color{red}\textbullet};
\node at (5.544232e-01,4.591433e+00) {\color{red}\textbullet};
\node at (1.464178e+00,4.026018e+00) {\color{red}\textbullet};
\node at (1.611964e+00,4.231395e+00) {\color{red}\textbullet};
\node at (1.911964e+00,3.831395e+00) {\color{red}\textbullet};
\node at (0.811964e+00,3.631395e+00) {\color{red}\textbullet};
\node at (1.211964e+00,3.431395e+00) {\color{red}\textbullet};
\node at (2.211964e+00,3.131395e+00) {\color{red}\textbullet};
\node at (0.65,2.4) {\color{red} $\star$};
\node at (0.45,1.35) {\color{red} $\star$};
\node at (1.25,1.55) {\color{red} $\star$};
\node at (1.05,1.82) {\color{red} $\star$};
\node at (1.55,2.5) {\color{red} $\star$};
\end{scope}
\begin{scope}[shift={(-0.8,2.16)},rotate=-60]
 \draw[thick,black] (0,0) rectangle (3,5);
\node at (6.693298e-01,4.106863e+00) {\color{blue}\textbullet};
\node at (1.163549e+00,3.174364e+00) {\color{blue}\textbullet};
\node at (1.656918e+00,3.648572e+00) {\color{blue}\textbullet};
\node at (5.544232e-01,4.591433e+00) {\color{blue}\textbullet};
\node at (1.464178e+00,4.026018e+00) {\color{blue}\textbullet};
\node at (1.611964e+00,4.231395e+00) {\color{blue}\textbullet};
\node at (0.611964e+00,3.031395e+00) {\color{blue}\textbullet};
\node at (0.911964e+00,3.431395e+00) {\color{blue}\textbullet};
\node at (1.011964e+00,3.731395e+00) {\color{blue}\textbullet};
\node at (1.111964e+00,4.331395e+00) {\color{blue}\textbullet};
\node at (1.45,2.4) {\color{blue} $\star$};
\node at (1.15,2.4) {\color{blue} $\star$};
\node at (1.65,2.7) {\color{blue} $\star$};
\node at (1.25,2.8) {\color{blue} $\star$};
\node at (1.75,2.5) {\color{blue} $\star$};
\end{scope}
\node at (0.35,1.2) {\color{blue} $\blacklozenge$};
\node at (0.15,0.90) {\color{blue} $\blacklozenge$};
\node at (0.39,0.7) {\color{blue} $\blacklozenge$};
\node at (0.00, 1.8) {\color{blue} $\star$};
\node at (0.41, 1.59) {\color{blue} $\star$};
\node at (0.32, 1.85) {\color{blue} $\star$};
\node at (0.52, 1.82) {\color{blue} $\star$};
\node at (0.92, 1.69) {\color{blue} $\star$};
\node at (0.294999e+00,43.151443e-01) {\color{violet}\textbullet};
\node at (0.177258e+00,39.203261e-01) {\color{violet}\textbullet};
\node at (1.416572e+00,39.266308e-01) {\color{violet}\textbullet};
\node at (0.323341e+00,36.106281e-01) {\color{violet}\textbullet};
\node at (1.117399e+00,40.018191e-01) {\color{violet}\textbullet};
\node at (0.042950e+00,35.370621e-01) {\color{violet}\textbullet};
\node at (1.842950e+00,42.370621e-01) {\color{violet}\textbullet};
\node at (0.442950e+00,36.370621e-01) {\color{violet}\textbullet};
\node at (1.142950e+00,44.370621e-01) {\color{violet}\textbullet};
\node at (0.942950e+00,36.870621e-01) {\color{violet}\textbullet}; 
\node at (-0.4,4.4) {$96$};
\node at (4,2.3) {\color{black} $t(n-24)$};
\node at (-2.5,3.5) {\color{black} $t(n-24)$};
\node at (0.35,2.18) {\color{black} $s_i(n-48)$};
\node at (4.70,1.35) {\color{black} $\Ss_1$};
\node at (-4.05,1.35) {\color{black} $\Ss_2$};
\end{tikzpicture}
}
\subcaption{$T(n, 96; t(n-24), t(n-24), 0)$ and ${\color{black}2 \times T(n-24, t(n-24); s_i(n-48), 0, 0)}$}
\label{fig_step3}
\end{minipage}\vspace{1em}

\begin{minipage}{0.95\textwidth}\centering
\scalebox{0.85}{
\begin{tikzpicture}[scale=.9]
\begin{scope}[shift={(0.0,-.45)},rotate=60]
 \draw[thick,black] (0,0) rectangle (3,5);
\node at (6.693298e-01,4.106863e+00) {\color{red}\textbullet};
\node at (1.163549e+00,3.174364e+00) {\color{red}\textbullet};
\node at (1.656918e+00,3.648572e+00) {\color{red}\textbullet};
\node at (5.544232e-01,4.591433e+00) {\color{red}\textbullet};
\node at (1.464178e+00,4.026018e+00) {\color{red}\textbullet};
\node at (1.611964e+00,4.231395e+00) {\color{red}\textbullet};
\node at (1.911964e+00,3.831395e+00) {\color{red}\textbullet};
\node at (0.811964e+00,3.631395e+00) {\color{red}\textbullet};
\node at (1.211964e+00,3.431395e+00) {\color{red}\textbullet};
\node at (2.211964e+00,3.131395e+00) {\color{red}\textbullet};
\end{scope}
\begin{scope}[shift={(-0.8,2.16)},rotate=-60]
 \draw[thick,black] (0,0) rectangle (3,5);
\node at (6.693298e-01,4.106863e+00) {\color{blue}\textbullet};
\node at (1.163549e+00,3.174364e+00) {\color{blue}\textbullet};
\node at (1.656918e+00,3.648572e+00) {\color{blue}\textbullet};
\node at (5.544232e-01,4.591433e+00) {\color{blue}\textbullet};
\node at (1.464178e+00,4.026018e+00) {\color{blue}\textbullet};
\node at (1.611964e+00,4.231395e+00) {\color{blue}\textbullet};
\node at (0.611964e+00,3.031395e+00) {\color{blue}\textbullet};
\node at (0.911964e+00,3.431395e+00) {\color{blue}\textbullet};
\node at (1.011964e+00,3.731395e+00) {\color{blue}\textbullet};
\node at (1.111964e+00,4.331395e+00) {\color{blue}\textbullet};
\end{scope}
\begin{scope}[shift={(-2.25,-1.1)}]
 \draw[thick,black] (0,0) rectangle (5,3);
\node at (4.294999e+00,3.151443e-01) {\color{violet}\textbullet};
\node at (4.177258e+00,1.203261e-01) {\color{violet}\textbullet};
\node at (4.616572e+00,4.266308e-01) {\color{violet}\textbullet};
\node at (4.323341e+00,6.106281e-01) {\color{violet}\textbullet};
\node at (4.817399e+00,4.018191e-01) {\color{violet}\textbullet};
\node at (4.042950e+00,5.370621e-01) {\color{violet}\textbullet};
\node at (3.842950e+00,2.370621e-01) {\color{violet}\textbullet};
\node at (3.442950e+00,6.370621e-01) {\color{violet}\textbullet};
\node at (3.642950e+00,4.370621e-01) {\color{violet}\textbullet};
\node at (3.942950e+00,6.870621e-01) {\color{violet}\textbullet};
\end{scope}
\node at (0.12, 2.5) {\color{blue} $\star$};
\node at (0.41, 2.39) {\color{blue} $\star$};
\node at (0.32, 2.55) {\color{blue} $\star$};
\node at (0.52, 2.52) {\color{blue} $\star$};
\node at (0.72, 2.39) {\color{blue} $\star$};
\node at (1.45,1.4) {\color{blue} $\star$};
\node at (1.15,0.4) {\color{blue} $\star$};
\node at (1.65,0.7) {\color{blue} $\star$};
\node at (1.25,0.8) {\color{blue} $\star$};
\node at (1.75,1.5) {\color{blue} $\star$};
\node at (-0.65,1.4) {\color{red} $\star$};
\node at (-0.15,0.35) {\color{red} $\star$};
\node at (-1.25,0.55) {\color{red} $\star$};
\node at (-1.05,0.82) {\color{red} $\star$};
\node at (-1.55,1.5) {\color{red} $\star$};
\node at (0.35,1.2) {\color{blue} $\blacklozenge$};
\node at (0.15,0.90) {\color{blue} $\blacklozenge$};
\node at (0.39,0.7) {\color{blue} $\blacklozenge$};
\node at (4,2.3) {\color{black} $96$};
\node at (-2.5,3.5) {\color{black} $96$};
\node at (2.4,-0.1) {\color{black} $96$};
\node at (0.35,1.65) {\color{black} $s_i(n-72)$};
\node at (0.35,2.08) {\color{black} $t(n-48)$};
\node at (1.99,1.1) {\color{black} $t(n-48)$};
\node at (-1.40,1.1) {\color{black} $t(n-48)$};
\node at (4.70,1.35) {\color{black} $\Ss_1$};
\node at (-4.05,1.35) {\color{black} $\Ss_2$};
\node at (3.20,-0.95) {\color{black} $\Ss_3$};
\end{tikzpicture}
}
\subcaption{$T(n, 0; 96, 96, 96)$ and ${\color{black} 3 \times T(n-24, 96; t(n-48), t(n-48), 0)}$}
\label{fig_step4}
\end{minipage}
\caption{A visualization of the first three steps in the inductive strategy. The blue points are $p_1, p_2, \ldots, p_{s_i(n-24)}$, the red points are $p_{s_i(n-24)+1}, \ldots, p_{s_i(n-24)+t(n-24)}$, and the collection of violet points contains $p_{s_i(n-24)+t(n-24)+1}, \ldots, p_{s_i(n)}$.} \label{fig_general_proof_strategy}
\end{figure}
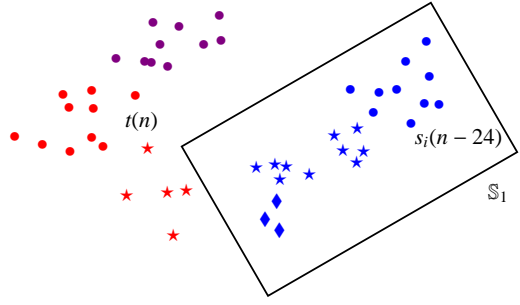
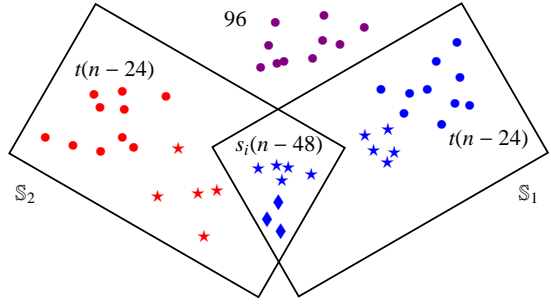
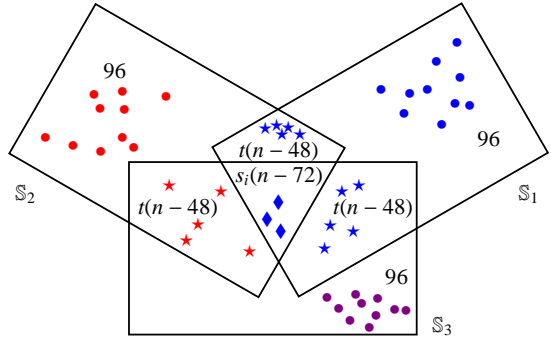

\paragraph{\textbf{Step 1}} 
Let $n \ge 32$. Then, we may specialize $s_i(n-24)$ out of $s_i(n)$ points, namely, the ``blue'' points $p_1$, $\ldots$, $p_{s_i(n-24)}$, to $\T_{n-24} = \T_n \cap \Ss_1$, leaving the remaining $t(n) = s_i(n) - s_i(n-24)$ points on $\T_{n} \setminus \Ss_1$ as illustrated in Figure~\ref{fig_general_proof_strategy}~(A). 

Let $A(n)$ be the vector space spanned by the hyperplanes of $\Ss$ containing the projectivization of $W(n,s_i(n);0,0,0)$, or, equivalently, the linear span of $T_{p_1}(\T_n)$, \dots, $T_{p_{s_i(n)}}(\T_n)$. Let $B(n)$ be the subspace of $A(n)$ spanned by the elements of $A(n)$ containing $\Ss_1$. Note that we can think of $B(n)$ as the vector space spanned by the hyperplanes of $\Ss$ containing the projectivization of $W(n,t(n);s_i(n-24),0,0)$. 

If $H_n \in A(n)\setminus B(n)$, then the intersection of $H_n$ with $\Ss_1$ is a hyperplane of $\Ss_1$, which we denote by $H_{n-24}$. By the definition of $A(n)$, the hyperplane $H_{n-24}$ contains the tangent space $T_{p_i}(\T_{n-24}) = T_{p_i}(\T_n) \cap \Ss_1$ to $\T_{n-24}$ at the blue points $p_i$ for each $i \in \{1, \dots, s_i(n-24)\}$. Since $H_{n-24}$ can be viewed as an element of $A(n-24)$, we obtain the following inequality: 
\[
\dim A(n) \leq \dim A(n-24) + \dim B(n). 
\]
This inequality is equivalent to the following inequality: 
\begin{eqnarray*}
\dim W(n,s_i(n);0,0,0) 
& \geq & \dim W(n,t(n);s_i(n-24),0,0) \\ 
& & + \dim W(n-24,s_i(n-24);0,0,0) - N(n-24). 
\end{eqnarray*}

As we shall see in Section~\ref{sec_inductive},  if $T(n,t(n);s_i(n-24),0,0)$ and $T(n-24,s_i(n-24);0,0,0)$ are true, then the above inequality is an equality. In particular, we obtain 
\[
\dim W(n,s_i(n);0,0,0) = w(n,s_i(n);0,0,0),
\]
which proves the truth of $T(n,s_i(n);0,0,0)$. Recall that the goal is proving that $T(n,3;s_i(n))$ is true for every $n \ge 8$ by induction. However, the specialization proposed here may only be applied for $n \ge 32$. Therefore, the initial cases of the induction, i.e., $n \in \{8,9,\ldots, 32\}$ should be proved separately; in Section \ref{sec_base}, a method is proposed for accomplishing this task. This first step reduces the problem of proving the truth of $T(n,3;s_i(n))$ to the problem of proving the truths of $T(n,t(n);s_i(n-24),0,0)$ for every $n \geq 32$.


\paragraph{\textbf{Step 2}} 
Let $n \ge 56$. Then, we may perform the following two specializations for verifying the truth of $T(n,t(n);s_i(n-24),0,0)$.
First, the $s_i(n-48)$ blue points $p_1, \dots, p_{s_i(n-48)} \in (\T_n \cap \Ss_1)$ are specialized to $(\T_n \cap \Ss_1) \cap \Ss_2$. The remaining $t(n-24)$ points are placed in $(\T_n \cap \Ss_1) \setminus \Ss_2$. Second, we specialize the $t(n-24)$ ``red'' points $p_{s_i(n-24)+1}, \dots, p_{s_i(n-24)+t(n-24)}$ to $(\T_n \cap \Ss_2) \setminus \Ss_1$. The remaining $96$ ``violet'' points are placed in $\T_n \setminus (\Ss_1 \cup \Ss_2)$. The resulting configuration of $s_i(n)$ points is illustrated in Figure~\ref{fig_general_proof_strategy}~(B). 

Let $C(n)$ be the subspace of $B(n)$ spanned by the hyperplane containing $\Ss_2$. One can verify that the tangent space to $\T_n$ at $p_j$ for each $j \in \{1, \dots, s_i(n-48)\}$ lies in the linear span of $\Ss_1$ and $\Ss_2$. In particular, 
\[
S_3(U_1) + S_3(U_2) + \sum_{j=1}^{s_i(n)} \widehat{T}_{p_j}(\T_n) = S_3(U_1) + S_3(U_2) + \sum_{j=s_i(n-48)+1}^{s_i(n)} \widehat{T}_{p_j}(\T_n).
\]
Therefore, $C(n)$ can be identified with the vector space of the hyperplanes in $\Ss$ containing the projectivization of $W(n,96;t(n-24),t(n-24),0)$. 

Let $H_n \in B(n) \setminus C(n)$. Then, the intersection of $H_n$ and $\Ss_2$ is a hyperplane in $\Ss_2$, which we denote by $H_{n-24}$. This hyperplane contains the following linear subspaces:
\begin{itemize}
 \item[--] the tangent spaces to the tangential variety $\T_{n-24} = \T_n \cap \Ss_2$ at the $t(n-24)$ red points $p_{s_i(n-24)+1}, \ldots, p_{s_i(n-24)+t(n-24)}$ that lie in $(\T_n \cap \Ss_2) \setminus \Ss_1$; 
 \item[--] the tangent spaces to $\T_{n-24}$ at the $s_i(n-48)$ blue points $p_1, \dots, p_{s_i(n-48)}$ in $(\T_n \cap \Ss_2) \cap \Ss_1$; and
 \item[--] the $24$-codimensional linear subspace $\Ss_1 \cap \Ss_2$ of $\Ss_2$.
\end{itemize}
Therefore, $H_{n-24}$ can be considered as an element of $B(n-24)$. This implies that 
\[
\dim B(n) \leq \dim B(n-24) +  \dim C(n),  
\]
or, equivalently, 
\begin{eqnarray*}
\dim W(n,t(n);s_i(n-24),0,0) 
& \geq & \dim W(n-24,t(n-24);s_i(n-48),0,0)  \\
& &  + \dim W(n,96;t(n-24),t(n-24),0) - N(n-24).
\end{eqnarray*}

In Section~\ref{sec_inductive}, we will show that the equality of the above inequality holds if $T(n,96;t(n-24),t(n-24),0)$ and $T(n-24,t(n-24);s_i(n-48),0,0)$ are true; consequently, $T(n,t(n);s_i(n-24),0,0)$ is true. The truth of the statement $T(n-24,t(n-24);s_i(n-48),0,0)$ for $n\in \{32, \dots, 55\}$, the base cases of the induction, will be proved in Section~\ref{sec_base}. 

\paragraph{\textbf{Step 3}}
Let $n \ge 80$. Then, we may perform the following specializations for proving the truth of the statement $T(n,96;t(n-24),t(n-24),0)$.
\begin{itemize} 
 \item[(i)] Specialize the $s_i(n-72)$ blue ``diamond'' points $p_1, \dots, p_{s_i(n-72)} \in (\T_n \cap \Ss_1 \cap \Ss_2)$ to $(\T_n \cap \Ss_1 \cap \Ss_2) \cap \Ss_3$. Leave the remaining $t(n-48)$ blue ``star'' points in $(\T_n \cap \Ss_1 \cap \Ss_2) \setminus \Ss_3$. 
 \item[(ii)] Specialize the $t(n-48)$ blue star points $p_{s_i(n-48)+1}, \dots, p_{s_i(n-48)+t(n-48)} \in (\T_n \cap \Ss_1)\setminus\Ss_2$ to $(\T_n \cap \Ss_1 \cap \Ss_3)\setminus\Ss_2$, leaving the remaining $96$ blue points in $(\T_n \cap \Ss_1) \setminus (\Ss_2 \cup \Ss_3)$.
\item[(iii)] Specialize the $t(n-48)$ red star points $p_{s_i(n-24)+1}, \dots, p_{s_i(n-24)+t(n-48)} \in (\T_n \cap \Ss_2)\setminus\Ss_1$ to $(\T_n\cap\Ss_2\cap\Ss_3)\setminus\Ss_1$. Leave the remaining $96$ red points in $(\T_n\cap\Ss_2) \setminus(\Ss_1\cup\Ss_3)$.
\item[(iv)] Specialize the $96$ violet points $p_{s_i(n-24)+t(n-24)+1}, \dots, p_{s_i(n)}$ to $(\T_n \cap \Ss_3)\setminus(\Ss_1\cup\Ss_2)$. 
\end{itemize}
The resulting configuration of $s_i(n)$ points now looks like Figure~\ref{fig_general_proof_strategy}~(C). 

Let $D(n)$ be the subspace of $C(n)$ that contains $\Ss_3$. As before, one can verify that the tangent spaces to~$\T_n$ at the star and diamond points $p_j$ for every $j$ in the union of $\{1, \dots, s_i(n-48)\}$, $\{s_i(n-48)+1, \dots, s_i(n-48)+t(n-48)\}$, and $\{s_i(n-24)+1, \dots, s_i(n-24)+t(n-48)\}$ are contained in the linear span of $\Ss_1$,  $\Ss_2$, and $\Ss_3$. Consequently, 
\[
\sum_{j=1}^3 S_3(U_j) + \sum_{j=1}^{s_i(n)} \widehat{T}_{p_j}(\T_n) = \sum_{j=1}^3 S_3(U_j) + \sum_{j \in \varLambda } \widehat{T}_{p_j}(\T_n),
\]
where $\varLambda$ is the union of $\{s_i(n-48)+t(n-48)+1, \dots,s_i(n-24)\}$,  $\{s_i(n-24)+t(n-48)+1, \dots,s_i(n-24)+t(n-24)\}$, and $\{s_i(n-24)+t(n-24)+1, \dots, s_i(n)\}$. Note that the points $p_j$ with $j \in \varLambda$ correspond to the ``regular'' circular points in Figure \ref{fig_general_proof_strategy}, i.e., the $96$ blue regular points, the $96$ red regular points, and the $96$ violet regular points. This means that $D(n)$ can be viewed as the vector space of hyperplanes of $\Ss$ that contain the projectivization of $W(n,0;96,96,96)$. 

Let $H_n \in C(n) \setminus D(n)$. Consider the hyperplane $H_{n-24}$ that is obtained as the intersection of $H_n$ and $\Ss_3$. This hyperplane contains the following linear subspaces: 
\begin{itemize}
 \item[--] the tangent spaces to the tangential variety $\T_{n-24} = \T_n \cap \Ss_3$ at the $96$ violet points $p_{s_i(n-24)+t(n-24)+1}, \dots, p_{s_i(n)} \in (\T_n\cap\Ss_3) \setminus (\Ss_1 \cup \Ss_2)$;
  \item[--] the tangent spaces to $\T_{n-24}$ at the $t(n-48)$ blue star points $p_{s_i(n-48)+1}$, $p_{s_i(n-48)+2}$, $\ldots$, $p_{s_i(n-48)+t(n-48)} \in (\T_n \cap \Ss_1 \cap \Ss_2) \setminus \Ss_3$;
  \item[--] the tangent spaces to $\T_{n-24}$ at the $t(n-48)$ red star points $p_{s_i(n-24)+1}$, $p_{s_i(n-24)+2}$, $\ldots$, $p_{s_i(n-24)+t(n-48)} \in (\T_n \cap \Ss_2 \cap \Ss_3)\setminus\Ss_1$; and 
  \item[--] the $24$-codimensional linear subspaces $\Ss_1 \cap \Ss_3$ and $\Ss_2 \cap \Ss_3$ of $\Ss_3$. 
\end{itemize}    
Therefore, $H_{n-24}$ can be considered as an element of $C(n-24)$, and thus we obtain the inequality
\[
\dim C(n) \leq  \dim C(n-24) + \dim D(n). 
\] 
This inequality is equivalent to the following inequality: 
\begin{eqnarray*}
 \dim W(n,96;t(n-24),t(n-24),0) & \geq & \dim W(n-24,96;t(n-48),t(n-48),0) \\
 && + \dim W (n,0;96,96,96) - N(n-24). 
\end{eqnarray*}
In Section~\ref{sec_inductive}, we will show that if $T(n,0;96,96,96)$ and $T(n-24,96;t(n-48),t(n-48), 0)$ are true, then the above inequality is an equality, and, hence, $\dim  W(n,96;t(n-24),t(n-24),0) = w(n,96;t(n-24),t(n-24),0)$.  
This establishes the truth of $T(n,96;t(n-24),t(n-24),0)$. 

The truths of the statements $T(n,96;t(n-24),t(n-24),0)$ for each $n \in \{56, \dots, 79\}$ will be shown in Section~\ref{sec_base}. Therefore, it remains only to show that $T(n,0;96,96,96)$ is true for every $n \geq 80$. 

\paragraph{\textbf{Step 4}}
Proving the truths of $T(n,0;96,96,96)$ for $n \ge 80$ can be reduced to showing the truth of $T(71,0;96,96,96)$ as follows. Let $W$ be a $72$-dimensional subspace of $U$ that intersects $U_1, U_2, U_3$, $U_1\cap U_2, U_1 \cap U_3, U_2 \cap U_3$, and $U_1 \cap U_2 \cap U_3$ properly. More precisely, $W$ satisfies the following conditions: 
\begin{itemize}
\item[--] $W\cap U_j$ is of codimension $24$ for each $j \in \{1,2,3\}$. 
\item[--] $W\cap U_j \cap U_k$ is of codimension $48$ for each $j, k\in \{1,2,3\}$. 
\item[--] $W\cap \bigcap_{j=1}^3 U_j$ is of codimension $72$.  
\end{itemize} 
As was discussed in Step 3, $T(n,0;96,96,96)$ is true if and only if the vector space 
\[
\sum_{j=1}^3 S_3(U_j) + \sum_{j \in \varLambda} T_{p_j}(\T_n)
\]
has dimension $w(n,0;96,96,96)$. We therefore no longer consider the star and diamond points, i.e., the $p_j$'s with $j \in \{1, \dots, s_i(n)\} \setminus \varLambda$. 
Instead, we only specialize
\begin{itemize}
\item[--] the $96$ regular blue points of $(\T_n \cap \Ss_1) \setminus (\Ss_2 \cup \Ss_3)$ to $\T_n \cap \P(S_3(W \cap U_1))$;
\item[--] the $96$ regular red points of $(\T_n \cap \Ss_2)\setminus(\Ss_1\cup\Ss_3)$ to $\T_n \cap \P(S_3(W \cap U_2))$; and
\item[--] the $96$ regular violet points of $(\T_n \cap \Ss_3)\setminus(\Ss_1\cup\Ss_2)$ to $\T_n \cap \P(S_3(W \cap U_3))$.
\end{itemize}

Let $E(n)$ be the subspace of $D(n)$ spanned by the hyperplanes containing $\P(S_3(W))$, and let $H_n \in D(n) \setminus E(n)$. Then, $H_n \cap \P(S_3(W))$ can be thought of as an element of $D(71)$ in the same way as in Steps 1, 2, and 3. We thus have the inequality 
\[
\dim D(n) \leq \dim D(71) + \dim E(n). 
\]
As shall be shown in Section~\ref{sec_inductive}, there is no hyperplane containing the linear space spanned by $\P(S_3(W))$, $\Ss_1$, $\Ss_2$, and $\Ss_3$, from which it follows that $\dim D(n) \leq \dim D(71)$. As is stated in Lemma~\ref{lem:equiabundant}, the vector space $D(n)$ is expected to have dimension~$0$ for every $n \ge 71$, or equivalently, the vector space $W(n,0;96,96,96)$ has dimension $N(n)$. Therefore, in order to prove the truths of $T(n,0;96,96,96)$ for every $n \geq 71$, it suffices to show that $T(71,0;96,96,96)$ is true.

Below we summarize Steps 1--4. 

\begin{proposition}
\label{prop:induction}
Let $i \in \{1,2\}$ and let $n \in \N$. 
\begin{itemize} 
 \item[(i)] Let $n \ge 32$. 
 If $T(n,t(n);s_i(n-24),0,0)$ and $T(n-24,s_i(n-24);0,0,0)$ are true, then so is~$T(n,s_i(n);0,0,0)$. 
 \item[(ii)] Let $n \ge 56$. If $T(n,96;t(n-24),t(n-24),0)$ and $T(n-24,t(n-24);s_i(n-48),0,0)$ are true, then so is~$T(n,t(n);s_i(n-24),0,0)$. 
 \item[(iii)] Let $n \ge 80$. If $T(n,0;96,96,96)$ and $T(n-24,96;t(n-48),t(n-48),0)$ are true, then so is~$T(n,96;t(n-24),t(n-24),0)$. 
 \item[(iv)] Let $n \ge 71$. If $T(71,0;96,96,96)$ is true, then so is $T(n,0;96,96,96)$. 
\end{itemize}
\end{proposition}

In order to finish the proof of Theorem~\ref{thm_main}, we need to verify the base cases, which are listed in the following corollary.
\begin{corollary} \label{cor_base_cases}
Let $i \in \{1,2\}$. If the following statements are true, then so is $T(n,3;s_i(n))$ for every $n \geq 8$: 
\begin{itemize}
 \item[(i)] $T(71,0;96,96,96)$; 
 \item[(ii)] $T(n,96;t(n-24),t(n-24),0)$ for every $n \in \{56, \ldots, 79\}$; 
 \item[(iii)] $T(n,t(n);s_i(n-24),0,0)$ for every $n \in \{32, \ldots, 55\}$; and
 \item[(iv)] $T(n,s_i(n);0,0,0)$ for every $n \in \{8, \ldots, 31\}$.
\end{itemize}
\end{corollary}

Combining the above corollary with the results of \cite{BCGI} proves Theorem \ref{thm_main}.

\section{The inductive cases: Proof of Proposition~\ref{prop:induction}} \label{sec_inductive}
The purpose of this section is to prove Proposition~\ref{prop:induction}. We will invoke some auxiliary lemma's to accomplish this task. The first of these states that the proposed specialization has some nice numerical features.
\begin{lemma}
\label{lem:equiabundant}
The statements $T(n, 0; 96, 96, 96)$ and $T(n, 96; t(n-24), t(n-24), 0)$ are equiabundant for $n \ge 71$ and $n \ge 48$ respectively.
\end{lemma}
\begin{proof}
From straightforward computations, it follows that 
\[
N(n) = \sum_{i=1}^3 (-1)^{i+1} \binom{3}{i} N(n-24i) + 3 \cdot 48 \cdot 96, 
\]
so that $T(n,0;96,96,96)$ is equiabundant. Note that we adopt the convention that $N(-1) = 0$ for the special case $n = 71$.

We also have 
\begin{align*}
N(n) - \sum_{i=1}^2 (-1)^{i+1} \binom{2}{i} N(n-24i) &= 576n-12672 \\ 
&=  576(24k+r)-12672 \\
&=  (9216k + 384r - 12768)+(4608k + 192r + 96) \\
&=  2\left(48\left(96k+4r-133\right)\right) + 96\left(2 \left(24k+r\right)+1\right) \\
&=  2\cdot 48\cdot  t(n-24)+96\left(2n + 1\right). 
\end{align*}
Therefore, $T(n,96;t(n-24),t(n-24),0)$ is equiabundant.
\end{proof}

Another nice feature of the specialization is revealed in the next lemma.
\begin{lemma}\label{lem_codim}
Let $n \ge 32$. Then, the expected codimension of $W(n, t(n); s_1(n-24), 0, 0)$ in~$S_3(U)$, denoted $c(n)$, is a function of the remainder of $n$ after division by $24$. In particular,~$c(n+24) = c(n)$.
\end{lemma}
\begin{proof}
 If $n \ge 32$, then the expected codimension of $W(n,t(n);s_1(n-24),0,0)$ in $S_3(U)$ is given explicitly by
\begin{align*}
 c(n) 
&= N(n) - w(n, t(n); s_1(n-24),0,0) \\
&= N(n) - N(n-24) - t(n)(2n+1) - 48 s_1(n-24) \\
&= 12 n^2 - 240 n + 1772 - (4n-37)(2n+1) - 48 s_1(n-24) \\
&= 4n^2 - 170 n + 1809 - 48^2 (k-1)^2 - 48(11+4r)(k-1) - 48 \lfloor (4r^2 + 22r +33)/48 \rfloor \\ 
&= 4r^2 -170 n + 1809 + 4608 k - 48^2 + 48(11+4r) -528 k - 48 \lfloor (4r^2 + 22r +33)/48 \rfloor \\ 
&= 4r^2 + 22 r + 33 - 48 \lfloor (4r^2 + 22r +33)/48 \rfloor,
\end{align*}
where $r$ and $k$ are the quotient and remainder by division of $n$ by 24. It follows that $c(n)$ is the remainder of $4r^2 +22r +33$ after division by $48$. 
\end{proof}

The final auxiliary result that we require is stated next.
\begin{lemma}
\label{lem:last-step}
Let $U_K$ be the subspace of $U$ spanned by $B_1 \cup B_2 \cup B_3$. Then, 
\[
S_3(U) = S_3(U_K) + \sum_{i=1}^3 S_3(U_i). 
\]
\end{lemma}
\begin{proof}
Let $E = \{x_i x_j x_k \, | \, 0 \leq i \leq j \leq k \leq n\}$ be a basis of $S_3(U)$. Then, it suffices to show that $E$ is contained in $S_3(U_K) + \sum_{l=1}^3 S_3(U_l)$. 
Obviously, the monomial $x_i x_j x_k$ with $0 \leq i \leq j \leq k \leq n$ is an element of $S_3(U_K)$ if and only if $k \le 71$. Hence, we may assume that $k \ge 72$. In this case, there is at least one $l \in \{1,2,3\}$ such that $x_i, x_j, x_k \in B \setminus B_l$, where $B_l$ is defined as in Section \ref{sec_convenient_subspaces}. This means that $x_i x_j x_k \in S_3(U_l)$ for such an $l$, concluding the proof. 
\end{proof}

\begin{proof}[Proof of Proposition~\ref{prop:induction} (i)]
Here, we only prove that if $T(n,t(n);s_1(n-24),0,0)$ and $T(n-24,s_1(n-24);0,0,0)$ are true, then so is $T(n,s_1(n);0,0,0)$.
The proof of the superabundant case follows along the same path.

Let $p_1, \dots, p_{s_1(n)} \in \T_n$ and let $Z = \left\{p_1^{2,3}, \dots, p_{s_1(n)}^{2,3}\right\}$. Then, we obtain the following short exact sequence:
\begin{equation}
\label{eq:seq1}
0 \rightarrow \sI_{L_1 \cup Z}(3) \rightarrow \sI_Z(3) \rightarrow \sI_{L_1 \cap Z, L_1}(3) \rightarrow 0. 
\end{equation}
Suppose that $p_1, \dots, p_{s_1(n-24)}$ are generic points of $\T_n \cap \Ss_1$ and that the remaining $t(n)$ points are generic points of $\T_n \setminus \Ss_1$. Then, by assumption and Remark~\ref{rem:hilbert-function-2},  
\[
 \dim H^0 \left(\P(U), \sI_{L_1 \cup Z}(3)\right)  =  N(n)-N(n-24)-(2n+1)t(n)-48s_1(n-24) 
\]
and 
\[
\dim H^0 \left(L_1, \sI_{L_1 \cap Z, L_1}(3)\right) = N(n-24)-s_1(n-24) (2(n-24)+1). 
\]
From~(\ref{eq:seq1}), it follows that
\begin{align*}
\dim H^0\left(\P(U), \sI_Z(3)\right) &\leq \dim  H^0 \left(\P(U), \sI_{L_1 \cup Z}(3)\right) + \dim H^0 \left(L_1, \sI_{L_1 \cap Z, L_1}(3)\right) \\
&= N(n)-s_1(n)(2n+1). 
\end{align*}
On the other hand, since $h_{\P (U)}(Z,3) \leq s_1(n)(2n+1)$ by~(\ref{eq:hilbert-function}), we have 
\[
\dim  H^0\left(\P(U), \sI_Z(3)\right) = N(n)-h_{\P(U)}(Z,3) \geq N(n) -s_1(n)(2n+1), 
\]
which implies that $\dim  H^0\left(\P(U), \sI_Z(3)\right) = N(n)-s_1(n)(2n+1)$. Therefore, the statement~$T(n,s_1(n);0,0,0)$ is true. 
\end{proof}

\begin{proof}[Proof of Proposition~\ref{prop:induction} (ii)]  
We first show that if $T(n,96;t(n-24),t(n-24),0)$ and $T(n-24,t(n-24);s_1(n-48),0,0)$ are true, then so is $T(n,t(n);s_1(n-24),0,0)$. 

%

Let $p_1, \dots, p_{t(n)} \in \T_n \setminus \Ss_1$, let $q_1, \dots, q_{s_1(n-24)} \in \T_n \cap \Ss_1$, and let 
\[
Z = \left\{p_1^{2,3},\dots, p_{t(n)}^{2,3}, \, q_1^{2,3}, \dots, q_{s_1(n-24)}^{2,3}\right\}.
\]  
Assume that $p_1, \dots, p_{t(n-24)}$ are generic points of $\T_n \cap \Ss_2$ and that $q_1, \dots, q_{s_1(n-48)}$ are generic points of $\T_n \cap \Ss_1 \cap \Ss_2$. Then, we obtain the following short exact sequence:
\begin{equation*}
0 \rightarrow \sI_{L_2 \cup Z}(3) \rightarrow \sI_Z(3) \rightarrow \sI_{L_2 \cap Z, L_2}(3) \rightarrow 0.  
\end{equation*}
By the assumption that the superabundant statement $T(n,96;t(n-24),t(n-24),0)$ is true, i.e., $\dim H^0\left(\P (U), \sI_{L_2 \cup Z}(3)\right) = 0$ by Lemma \ref{lem:equiabundant}, and since $T(n-24,t(n-24);s_1(n-48),0,0)$ is true by assumption, we obtain
\begin{align*}
\dim H^0\left(\P(U), \sI_Z(3)\right) 
&\leq \dim H^0\left(\P (U), \sI_{L_2 \cup Z}(3)\right) + \dim H^0\left(L_2, \sI_{L_2\cap Z, L_2}(3)\right) \\
&= 0 + c(n-24) = c(n). 
\end{align*}
The last equality is due to Lemma~\ref{lem_codim}.
On the other hand, the actual codimension of $H^0\left(\P(U), \sI_Z(3)\right)$ is bigger than or equal to $c(n)$. Therefore, $\dim H^0\left(\P(U), \sI_Z(3)\right) = c(n)$, and hence $T(n,t(n);s_1(n-24),0,0)$ is true.  

Note that the statement $T(n,t(n);s_2(n-24),0,0)$ is superabundant. Consequently, the expected codimension of $W(n,t(n);s_2(n-24),0,0)$ is zero. We can therefore prove that the truths of $T(n,96;t(n-24),t(n-24),0)$ and $T(n-24,t(n-24);s_2(n-48),0,0)$ imply that~$T(n,t(n);s_2(n-24),0,0)$ is true in the same way as for $T(n,t(n);s_1(n-24),0,0)$. 
\end{proof}

\begin{proof}[Proof of Proposition \ref{prop:induction} (iii)]
For each $i \in \{1,2\}$, let $p_{i,1}, \dots, p_{i,t(n-24)} \in \T_n \cap \Ss_i$, let $q_1$, $\ldots$, $q_{96} \in \T_n \setminus (\Ss_1 \cup \Ss_2)$, and let 
\[
Z = \left\{p_{1,1}^{2,3}, \dots, p_{1,t(n-24)}^{2,3}, \, p_{2,1}^{2,3}, \dots, p_{2,t(n-24)}^{2,3}, \, q_1^{2,3}, \dots, q_{96}^{2,3}\right\}. 
\]
Assume that 
\begin{itemize}
\item[(1)] $p_{i,1}, \dots, p_{i,t(n-48)}$ are generic points of $(\T_n \cap \Ss_i) \cap \Ss_3$; 
\item[(2)] the $96$ points $p_{i,t(n-48)+1}, \ldots, p_{i,t(n-24)}$ are generic points of $(\T_n \cap \Ss_i) \setminus \Ss_3$; and 
\item[(3)] $q_1, \dots, q_{96}$ are generic points of $\T_n \cap \Ss_3$. 
\end{itemize}
Then, we have the short exact sequence 
\begin{equation}
\label{eq:seq3}
0 \rightarrow \sI_{L_3 \cup Z}(3) \rightarrow \sI_Z(3) \rightarrow \sI_{L_3 \cap Z, L_3}(3) \rightarrow 0.  
\end{equation}
The assumption that the statements $T(n,0;96,96,96)$ and $T(n-24,96;t(n-48),t(n-48),0)$ are true implies 
\[
\dim H^0\left(\P(U), \sI_{L_3 \cup Z}(3) \right) = \dim H^0\left(L_3, \sI_{L_3 \cap Z, L_3}(3)\right) = 0 
\]
by Lemma \ref{lem:equiabundant}.
Taking cohomology of~(\ref{eq:seq3}) gives rise to 
\[
\dim H^0 \left(\P(U), \sI_Z(3)\right) \leq \dim H^0\left(\P(U), \sI_{L_3 \cup Z}(3) \right) + \dim H^0\left(L_3, \sI_{L_3 \cap Z, L_3}(3)\right)  = 0,  
\]
from which it follows that $\dim H^0 \left(\P(U), \sI_Z(3)\right) =0$. Hence, $T(n,96;t(n-24),t(n-24),0)$ is true. 
\end{proof}

\begin{proof}[Proof of Proposition \ref{prop:induction} (iv)]
For each $i \in \{1,2,3\}$, let $p_{i,1}, \dots, p_{i,96}$ be generic points of $\T_n \cap \Ss_i$. Let 
\[
Z = \left\{\left. p_{i,1}^{2,3}, \dots, p_{i,96}^{2,3} \, \right| \, i \in \{1,2,3\}\right\}. 
\]
For simplicity, denote $\bigcup_{i=1}^3 L_i$ by $L$. Let $U_K$ be the subspace of $U$ spanned by $B_1 \cup B_2 \cup B_3 = \{x_0, \ldots, x_{71}\}$, and let $K = \P(U_K)$. 
Assume that $p_{i,1}, \dots, p_{i,96} \in \T_n \cap \P\left(S_3(U_K)\right)$ for each $i \in \{1,2,3\}$. Then, we have the following short exact sequence:
\begin{equation}
\label{eq:seq4}
 0 \rightarrow \sI_{L\cup Z \cup K}(3) \rightarrow \sI_{L \cup Z}(3) \rightarrow \sI_{(L\cup Z)\cap K, K}(3) \rightarrow 0. 
\end{equation}
Since $\dim H^0\left(\P(U), \sI_{L \cup K}(3) \right) = 0$ by Lemma \ref{lem:last-step} and 
\[
\dim H^0\left(\P(U), \sI_{L \cup K}(3) \right) \geq \dim H^0 \left(\P(U), \sI_{L \cup Z \cup K}(3)\right), 
\]
we get $ \dim H^0 \left(\P(U), \sI_{L \cup Z \cup K}(3)\right)=0$. Furthermore, the assumption that the statement $T(71,0;96,96,96)$ is true implies that $\dim H^0\left(K,\sI_{(L\cup Z)\cap K, K}(3)\right) =0$ by Lemma \ref{lem:equiabundant}. Taking cohomology of~(\ref{eq:seq4}) therefore yields $\dim H^0\left(\P (U), \sI_{L\cup Z}(3)\right) =0$. Thus, the statement $T(n,0;96,96,96)$ is also true, and hence we completed the proof. 
\end{proof}

\section{The base cases: Proof of Theorem \ref{thm_main}} \label{sec_base}
It remains to prove the correctness of the base cases of the inductive proof presented in the previous section, i.e., those cases appearing in Corollary \ref{cor_base_cases}. For proving the truths of the statements $T(n, s; a_1, a_2, a_3)$, we propose constructing a matrix $\mathbf{T}$ whose column span coincides with the subspace $W(n, s; a_1, a_2, a_3)$. The rank of $\mathbf{T}$ then coincides with the dimension of $W(n, s; a_1, a_2, a_3)$. Adopting such an approach allows us to leverage efficient algorithms from linear algebra, which were already employed with success in the context of identifiability of the tensor rank decomposition \cite{COV2014} and Waring's decomposition \cite{COV2015}.

Throughout this section, we deal with matrix representations of linear spaces, meaning that we will work in coordinates. We adopt the monomial basis $\{x_0, x_1, \ldots, x_n\}$ of $U \cong \C^{n+1}$. 
A natural basis of $S_3(U)$, when considered as the degree three piece of $\C[x_0, x_1, \ldots, x_n]$, is $E = \{ x_i x_j x_k \;|\; 0 \le i \le j \le k \le n \}$. By considering the lexicographic total monomial order $\le_\text{lex}$ on the elements of $E$, we can define an embedding 
\begin{align*}
 \nu_3:\qquad E &\to \C^{\binom{n+3}{3}} \\
x_i x_j x_k &\mapsto e_z,
\end{align*}
where $z+1$ is the position of $x_i x_j x_k$ in the lex-ordered sequence $x_0^3$, $x_0^2 x_1$, $x_0^2 x_2$, $\ldots$, $x_{n-1} x_n^2$, $x_n^3$, and $e_k$ is the standard basis vector that has a $1$ in position $k+1$ and zeros elsewhere. The domain of the map can be extended to the whole of $S_3(U)$ by linearity, namely define it as $\nu_3(w) = \sum_{0\le i\le j\le k \le n} c_{i,j,k} \nu_3(x_i x_j x_k)$ for $w = \sum_{0\le i\le j\le k \le n} c_{i,j,k} x_i x_j x_k$. With regard to these bases, the product of the linear forms
\begin{align*}
k_i &= k_{i,0} x_0 + k_{i,1} x_1 + \cdots + k_{i,n} x_n, \\
\ell_i &= \ell_{i,0} x_0 + \ell_{i,1} x_1 + \cdots + \ell_{i,n} x_n, \mbox{ and } \\
m_i &= m_{i,0} x_0 + m_{i,1} x_1 + \cdots + m_{i,n} x_n,
\end{align*}
i.e., $k_i \ell_i m_i \in S_3(U)$, is represented explicitly by the vector
\begin{align*}
\nu_3( k_i \ell_i m_i ) =
\begin{bmatrix}
k_{i,0} \ell_{i,0} m_{i,0} \\
k_{i,0}\ell_{i,0}m_{i,1} + k_{i,0}\ell_{i,1}m_{i,0} + k_{i,1}\ell_{i,0}m_{i,0} \\
\vdots \\
k_{i,n-1}\ell_{i,n-1}m_{i,n} + k_{i,n-1}\ell_{i,n}m_{i,n-1} + k_{i,n}\ell_{i,n-1}m_{i,n-1} \\
k_{i,n}\ell_{i,n}m_{i,n}
\end{bmatrix}
\end{align*}
with respect to the standard basis of $\C^{\binom{n+3}{3}}$.

\subsection{The algorithm}
Having stated accurately the isomorphism between $S_3(U)$ and $\C^{\binom{n+3}{3}}$, we can propose a basic approach for checking the truth of $T(n, s; a_1, a_2, a_3)$. It is presented as Algorithm~\ref{alg_basic_algorithm}.

\begin{algorithm}
\begin{enumerate}
\item
Select $s$ points $p_1, \dots, p_s$ of $\T_n$ by randomly choosing $\ell_i, m_i \in U$. 
As was shown in (\ref{eq:tangentSpace}), the affine cone over the tangent space to $\T_n$ at $p_i$ is given explicitly by
\(
 \ell_i^2 U+ \ell_i m_i U. 
\)
We can represent this space practically by the column span of the matrix
\[
 \mathbf{T}_{p_i} = \begin{bmatrix} \nu_3(\ell_i^2 x_0) & \cdots & \nu_3(\ell_i^2 x_n) & \nu_3(\ell_i m_i x_0) & \cdots & \nu_3(\ell_i m_i x_n) \end{bmatrix}.
\]
Let $\mathbf{R}_1$ be the horizontal concatenation of the matrices just defined.

\item For every $a_l \ne 0$, $l = 1,2,3$, we choose $a_l$ points $q_{l,1}, \dots, q_{l,a_l}$ of $\T_n \cap \Ss_l$ by randomly drawing two vectors $\ell_{l,j}, m_{l,j} \in U_l$. The span of the affine cone over the tangent space to $\T_n$ at the $q_{l,j}$'s modulo $S_3(U_l)$ is given by
\(
\ell_{l,j}^2 \overline{U}_l + \ell_{l,j} m_{l,j} \overline{U}_l. 
\) 
We can represent this space practically by the column span of the matrix
\begin{align*}
\mathbf{T}_{q_{l,j}} &= \begin{bmatrix} \mathbf{T}_{q_{l,j},1} & \mathbf{T}_{q_{l,j},2} \end{bmatrix}, \quad\text{where} \\
\mathbf{T}_{q_{l,j},1} &= \begin{bmatrix} \nu_3(\ell_{l,j}^2 x_{24(l-1)}) & \cdots & \nu_3(\ell_{l,j}^2 x_{24l-1}) \end{bmatrix} \text{ and} \\
\mathbf{T}_{q_{l,j},2} &= \begin{bmatrix} \nu_3(\ell_{l,j} m_{l,j} x_{24(l-1)}) & \cdots & \nu_3(\ell_{l,j} m_{l,j} x_{24l-1}) \end{bmatrix}.
\end{align*}
Let $\mathbf{R}_2$ be the horizontal concatenation of the matrices $\mathbf{T}_{q_{l,j}}$ just defined.

\item For every nonzero $a_l$, $l=1,2,3$, construct the matrix
\[
 \mathbf{F}_l = \begin{bmatrix} \nu_3(x_i x_j x_k) \;|\:\; 0 \le i \le j \le k \le n,\; i,j,k \not\in \{ 24(l-1), \ldots, 24l-1 \} \end{bmatrix},
\]
whose column span coincides with $S_3(U_l)$. Let $\mathbf{F}$ be the matrix obtained from concatenating all $\mathbf{F}_l$ horizontally. 

\item
Concatenate the matrices $\mathbf{F}$, $\mathbf{R}_{1}$, and $\mathbf{R}_{2}$ horizontally, and call the result $\mathbf{T}$. By construction, the rank of $\mathbf{T}$ coincides with the dimension of $W(n, s; a_1, a_2, a_3)$. If the rank is maximal, i.e., equal to the expected value (\ref{eq:dimension}), then $T(n,s;a_1,a_2,a_3)$ is true by semicontinuity. Otherwise, the algorithm declares that it does not know the answer; either the chosen points $p_i$ and $q_{i,j}$ were not sufficiently general, or $T(n, s; a_1, a_2, a_3)$ is false.
\end{enumerate}
\caption{A simple algorithm for confirming the truth of $T(n, s; a_1, a_2, a_3)$.}
\label{alg_basic_algorithm}
\end{algorithm}

While Algorithm~\ref{alg_basic_algorithm} is mathematically correct, implementing it as such may lead to a huge computational cost. As an example, consider the application of the algorithm to the statement $T(79, 96; 183, 183, 0)$, which corresponds to the most challenging case that we should prove in Corollary~\ref{cor_base_cases}. In the first step, $96$ matrices of size $88560 \times 160$ are constructed. The second step would construct two sets of $183$ matrices of size $88560 \times 48$. In the third step, the algorithm constructs $\mathbf{F}_1$, which is of size $N(79) \times N(55) = 88560 \times 30856$, and $\mathbf{F}_2$, which is of the same size. The concatenation of all these matrices is thus of size $88560 \times 94640$. Simply storing this matrix using standard $64$-bit integers would result in a memory consumption of about $62.4$GB. Based on our computational experiments, we can retroactively estimate that computing the rank of this matrix would take about two and a half days using one processing unit.

For overcoming the aforementioned double computational hurdle of memory consumption and long computation time, we propose a simple trick that will greatly improve the computational characteristics of the cases (ii) and (iii) of Corollary~\ref{cor_base_cases}, especially for the larger values of $n$. Naturally, one understands from Lemma~\ref{lem_basis} that the matrices $\mathbf{F}_i$ will contain some identical columns and that they may be removed from $\mathbf{T}$ without affecting the latter's rank. However, the gain from this will only be marginal; in the preceding example, one would still have to compute the rank of a $88560 \times 88656$ matrix. The more noteworthy improvement that we suggest is based on the following elementary property of orthogonal projectors; let $\mathbf{A} \in \C^{m \times n}$ and $\mathbf{B} \in \C^{m \times p}$, and then
\begin{align}\label{eqn_rank_help}
\rank{ \begin{bmatrix} \mathbf{A} & \mathbf{B} \end{bmatrix} } 
= \rank{\mathbf{A}} + \rank{ \mathbf{P}_{\mathbf{A}}^\perp \mathbf{B} } = n + \rank{ \mathbf{P}_{\mathbf{A}}^\perp \mathbf{B} },
\end{align}
provided that $\mathbf{A}$ is a matrix with $n \le m$ linearly independent columns, and where $\mathbf{P}_{\mathbf{A}}^\perp = \mathbf{I}_n - \mathbf{A}(\mathbf{A}^* \mathbf{A})^{-1}\mathbf{A}^*$ is the projection onto the orthogonal complement of the span of the columns of $\mathbf{A}$; herein, $\mathbf{A}^*$ is the conjugate transpose of $\mathbf{A}$, and $\mathbf{I}_n$ is the $n \times n$ identity matrix. 

Let us partition $\mathbf{T}$ as
\(
 \mathbf{T} = \left[\begin{smallmatrix} \mathbf{F} & \mathbf{R} \end{smallmatrix}\right],
\)
and let $\mathbf{F}'$ be the matrix obtained from $\mathbf{F}$ by removing identical columns, i.e., the columns of $\mathbf{F}'$ constitute a basis of the column space of $\mathbf{F}$. Then, the column span of $\mathbf{T}$ equals the column span of $\mathbf{T}' = \left[\begin{smallmatrix} \mathbf{F}' & \mathbf{R} \end{smallmatrix}\right]$. We will now apply (\ref{eqn_rank_help}) to $\mathbf{T}'$. 
With the proposed isomorphism $\nu_3$ it can be verified that the $t$ columns of $\mathbf{F}'$ form a subset of the standard basis vectors of $\C^{\binom{n+3}{3}}$. Let us write $\mathbf{F}'$ explicitly as $\mathbf{F}' = \left[\begin{smallmatrix} e_{z_1} & e_{z_2} & \cdots & e_{z_t} \end{smallmatrix}\right]$, where $0 \le z_1 < z_2 < \cdots < z_t \le \binom{n+3}{3}$ and where the particular values of $z_i$ are those implicitly given in the proof of Lemma~\ref{lem_basis}. In this case, the projection onto the complement of the column span of $\mathbf{F}'$ takes a particularly pleasing form: 
\[
\mathbf{P}_{\mathbf{F}'}^\perp = \mathbf{I}_n - \mathbf{F}' (\mathbf{I}_t)^{-1} (\mathbf{F}')^* = \mathbf{I}_n - \sum_{i=1}^t e_{z_i} e_{z_i}^*.
\]
Thus, $\mathbf{P}_{\mathbf{F}'}^\perp \mathbf{R}$ simply sets the $z_1$th, $z_2$th, $\ldots$, $z_t$th row of $\mathbf{R}$ to zero. Since the rows consisting only of zeros do not influence the rank of a matrix, one may even remove them. Let $Y = \{1,2,\ldots,N(n)\} \setminus \{z_1,z_2,\ldots,z_t\}$ consist of the indices of the rows that are not zero. Then, we denote by $\mathbf{R}(Y)$ the matrix consisting of the rows of $\mathbf{R}$ with row indices appearing in $Y$.
As a result, we have
\[
\rank \mathbf{T} = \rank \mathbf{T}' = \dim \sum_{j=1}^3 \chi(a_j) S_3(U_j) + \rank \mathbf{R}(Y),
\]
where $\chi(a_j) S_3(U_j) = \emptyset$ if $a_j=0$, and $S_3(U_j)$ otherwise. In the foregoing, the value of the first term is given explicitly by Lemma \ref{lem_basis}. Consequently, it suffices to compute the rank of the subset of the rows of $\mathbf{R}$ corresponding to the index set $Y$. It should be remarked that the elements of $Y$ can be computed as the complement of the set of $z_i$'s whose values can be computed directly from the definition of the map $\nu_3$; one simply computes the position of a monomial in some lex-ordered sequence of monomials. In particular, this calculation can be performed without auxiliary memory requirements. That is, we only need to store $|Y|$ integer values, which contrasts markedly with the straightforward approach in which the entire matrix $\mathbf{F}$ would be constructed explicitly.

In conclusion, we propose optimizing the computational properties of the basic algorithm by replacing steps (3) and (4) with the following alternative:
\begin{itemize}
 \item[(3)] For every nonzero $a_l$, compute the set of integers
\[
 Z_l = \bigl\{ z \;|\; e_z = v_3(x_i x_j x_k),\; i,j,k \in \{0,1,\ldots,n\}\setminus\{24(l-1),\ldots,24l-1\} \bigr\},
\]
while $Z_l = \emptyset$ for every $a_l$ that is zero. Compute 
\[
Y = \bigl\{ 1,2,\ldots,N(n) \bigr\} \setminus (Z_1 \cup Z_2 \cup Z_3).
\]
 
 \item[(4)] Concatenate the matrices $\mathbf{R}_1$ and $\mathbf{R}_2$ horizontally, and call the resulting matrix $\mathbf{R}$. By construction, the rank of $\mathbf{R}(Y)$ coincides with 
\[
 \dim W(n, s; a_1, a_2, a_3) - \dim \sum_{i=1}^3 \chi(a_j) S_3(U_j).
\]
If the rank of $\mathbf{R}(Y)$ plus $\dim \sum_{i=1}^3 \chi(a_j) S_3(U_j)$ equals the maximal value, i.e., the expected value (\ref{eq:dimension}), then $T(n, s; a_1, a_2, a_3)$ is true by semicontinuity. Otherwise, the algorithm declares that it does not know the answer; either the chosen points $p_i$ and $q_{i,j}$ were not sufficiently general, or $T(n, s; a_1, a_2, a_3)$ is false.
\end{itemize}
We will refer to this version of Algorithm~\ref{alg_basic_algorithm} as the optimized version, while the original statement of Algorithm~\ref{alg_basic_algorithm} will be called the basic version.

\subsection{Implementation aspects}
For efficiently and reliably computing the rank of $\mathbf{T}$, we propose to adopt the same strategy as \cite{COV2014}. The idea consists of choosing the points $p_i$ and $q_{i,j}$ in a convenient manner. Let $a$ be a sufficiently large prime number and consider the definitions of $U$ and $U_l$ over the prime field $\Z_a$. We suggest picking $p_i = \ell_i^{d-1} m_i$ by sampling random vectors $\ell, m \in U$ whose entries are uniformly drawn from $[0, a-1]$. Similarly, we choose $q_{i,j}$ by randomly sampling two vectors $\ell, m \in U_i$. We may then obtain a lower bound on the rank of $\mathbf{T}$ by computing the rank of $\mathbf{T}$ over $\Z_a$, rather than over $\C$. The rank of $\mathbf{T}$ over the finite field $\Z_a$ may be computed by reducing $\mathbf{T}$ to row echelon form by applying Gaussian elimination over $\Z_a$.
The proposal to bound the rank from below by choosing special points and performing a rank computation over a finite field does not fundamentally alter the output of the optimized version of Algorithm~\ref{alg_basic_algorithm}. When the rank of $\mathbf{T}$ is maximal over $\Z_a$, i.e., equal to the expected value $w(n, s; a_1, a_2, a_3)$, then the rank over $\C$ will also be maximal. If the rank of $\mathbf{T}$ is less than expected, then the algorithm will claim that it does not know whether $T(n, s; a_1, a_2, a_3)$ is true; our modification has only introduced an additional source of uncertainty, namely it may be the case that the rank over $\Z_a$ is strictly strictly less than over $\C$.

The optimized algorithm was implemented in C++, and it is included in the ancillary files. We used basic data structures from the Eigen v3 \cite{Eigen} library. The construction of $\mathbf{R}$ was partially parallelized with OpenMP v3.1. The rank $\mathbf{R}(Y)$ was computed over a finite field with characteristic $8191$ using the $\texttt{Rank}$ function provided by FFLAS--FFPACK \cite{FFLAS}, which essentially computes an $LU$-factorization with row and column pivoting. The underlying BLAS implementation that FFLAS--FFPACK requires was the optimized OpenBLAS \cite{OpenBLAS} library. We ran the program on a computer system containing $128$GB of main memory and two central processing units (twice an Intel Xeon E5-2697 v3), each with $14$ processing cores clocked at $2.60$GHz. Using \texttt{numactl} we instructed the software to use all $14$ cores of one central processing unit; in particular, the calculation of the rank over $\Z_{8191}$ proceeded in parallel.

\subsection{Computational complexity} 
As we claimed before, pursuing the optimization that was proposed for the basic version of Algorithm~\ref{alg_basic_algorithm} is worthwhile for proving statements of type (ii) and (iii) in Corollary \ref{cor_base_cases}. We can now justify this claim.

\begin{proposition}\label{prop_statement_ii}
The time complexity of the optimized version of Algorithm~\ref{alg_basic_algorithm} for verifying the truth of the statement $T(n, 96; t(n-24), t(n-24), 0)$ is \(\ko(n^4)\). Its space complexity is $\ko(n^4)$.
\end{proposition}
\begin{proof}
The time complexity of the first two steps of Algorithm~\ref{alg_basic_algorithm} amounts to 
\[
 \ko\bigl( 96(2n+2) N(n) + 2\cdot48\cdot t(n-24) N(n) \bigr) = \ko( n^4 )
\]
operations for constructing the matrix $\mathbf{R}$; recall that $t(n) = 4n - 37$ scales linearly in $n$. 
The optimized algorithm retains only
\[
 |Y| = N(n) - 2 N(n-24) + N(n-48) = 576n - 12672
\]
rows of $\mathbf{R}$ in the last step, while the number of columns of $\mathbf{R}$ is 
\[
2 \cdot t(n-24) \cdot 48 + 96 (2n+2) = 576n - 12576.
\] 
Hence, the Gaussian elimination step for computing the rank of $\mathbf{R}(Y)$ requires only $\ko(n^3)$ operations. The time complexity is thus dominated by the cost of constructing $\mathbf{R}$. 
From the size of the involved matrices, it follows that $\mathbf{R}$ is the largest matrix that should be (temporarily) stored, for a total space complexity of $\ko(n^4)$ values in $\Z_a$.
\end{proof}

\begin{proposition}\label{prop_statement_iii}
 The time complexity of the optimized version of Algorithm~\ref{alg_basic_algorithm} for verifying the truth of the statement $T(n, t(n); s_i(n-24), 0, 0)$, $i\in\{1,2\}$, is \(\ko(n^6)\). Its space complexity is $\ko(n^5)$.
\end{proposition}
\begin{proof}
The time complexity of the first two steps is
\[
 \ko\bigl( t(n) (2n+2) N(n) + 48 s_i(n-24) N(n)  \bigr) = \ko(n^5);
\]
recall from the explicit expression in Section \ref{sec_sub_sup} that $s_i(n)$ scales quadratically in $n$. The number of rows of $\mathbf{R}(Y)$ equals $N(n) - N(n-24) = \ko(n^2)$, while the number of columns equals $t(n)(2n+2) + 48s_i(n-24) = \ko(n^2)$. Therefore, the rank computation requires $\ko(n^6)$ operations, yielding the time complexity. From the size of the involved matrices, it follows that storing $\mathbf{R}$ requires $\ko(n^5)$ values of $\Z_a$.
\end{proof}

Statements of type (ii) and (iii) in Corollary \ref{cor_base_cases} can thus be verified in a more efficient manner than attempting to  prove the statement $T(n, s_i(n); 0,0,0)$, $i\in\{1,2\}$, directly. It is namely straightforward to verify the following result.
\begin{proposition}\label{prop_statement_iv}
 The time complexity of both versions of Algorithm~\ref{alg_basic_algorithm} for verifying the truth of the statement $T(n, s_i(n); 0, 0, 0)$, $i\in\{1,2\}$, is \(\ko(n^9)\). Its space complexity is $\ko(n^6)$.
\end{proposition}

\begin{figure}
\input{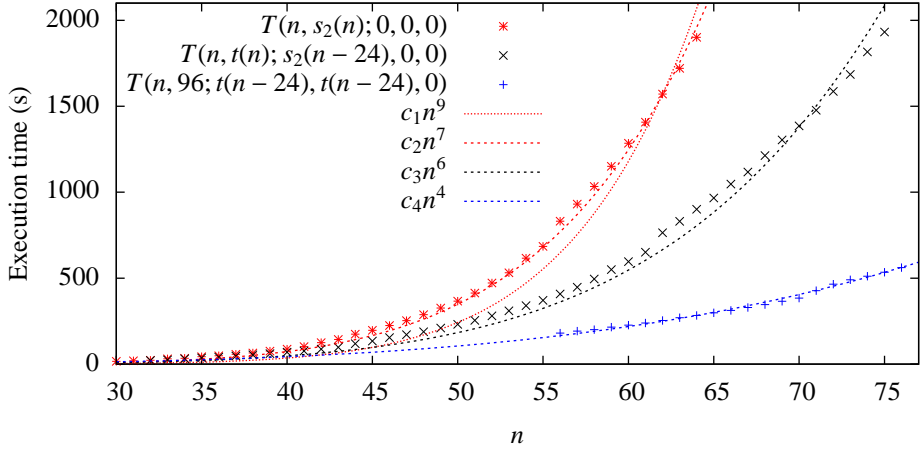}
\caption{A plot of the execution time for proving the different types of statements in function of the dimension $n$ of $U$. The dotted lines indicate the asymptotic time complexity, as determined in Propositions \ref{prop_statement_ii}, \ref{prop_statement_iii} and \ref{prop_statement_iv}.}
\label{fig_computation}
\end{figure}

As a confirmation of foregoing time complexity analyses, we will compare them with experimentally obtained execution times of the optimized implementation of Algorithm~\ref{alg_basic_algorithm}. 
In Figure \ref{fig_computation}, we plot the total execution time of the program for verifying the truth of the three types of statements featured in the foregoing propositions for increasing values of the dimension $n$ of the projective space $\P(U)$. In the figure, we also display four monomials that were fitted to the data: the monomials $c_1 n^9$ and $c_2 n^7$ were fitted to the execution times for verifying statements of type $T(n, s_2(n); 0, 0, 0)$, the monomial $c_3 n^6$ was fitted to the execution times for verifying statements of type $T(n, t(n); s_2(n-24), 0, 0)$, and the monomial $c_4 n^4$ was fitted to the execution times for verifying the statements $T(n, 96; t(n-24), t(n-24), 0)$. The constants $c_i$ of the monomials were determined using the \texttt{fit} command of gnuplot 4.6. As can be seen, the experimental results for the statements appearing in Propositions \ref{prop_statement_ii} and \ref{prop_statement_iii} line up well with the theoretically determined asymptotic time complexities. However, the execution times for proving statements $T(n, s_2(n); 0, 0, 0)$ do not seem to obey an asymptotic growth of $\ko(n^9)$; rather, it seems to grow only like $\ko(n^7)$. One likely explanation for this phenomenon is that the pivoted $LU$-factorization method in FFLAS--FFPACK is exploiting the additional structure that is present in $\mathbf{R}$. The matrix $\mathbf{R}(Y) = \mathbf{R}$ corresponding to statements $T(n, s_i(n); 0, 0, 0)$ is namely very \emph{sparse}, i.e., it contains many entries equal to zero. Consider the first step of Algorithm~\ref{alg_basic_algorithm}, and then it is clear that both $\nu_3(\ell_i^2 x_j)$ and $\nu_3(\ell_i m_i x_j)$ each contain precisely $\dim S_2(U) = \binom{n+2}{2}$ nonzero values, for a generic choice of $\ell_i$ and $m_i$. Hence, out of the $2(n+1)\binom{n+3}{3}$ elements of $\mathbf{T}_{p_i}$ only $2(n+1)\binom{n+2}{2}$ of them are not equal to zero. That is, the fraction of zeros of $\mathbf{T}_{p_i}$, and, hence, of $\mathbf{T}$, equals $1 - 3(n+3)^{-1}$, which tends to $1$ as $n \to \infty$. It may thus be appropriate to treat $\mathbf{R}$ as a sparse matrix when computing its rank. The design and implementation of efficient methods for sparse $LU$-factorization is a field with a very rich literature (see, e.g., \cite{Davis,Duff}), but it lies well beyond the scope of the present paper. 

Figure \ref{fig_computation} facilitates appreciation of the tremendous difference in execution times for projective spaces of equal dimension when proving one of the three types of statements. 
That is, the optimizations that we proposed reduce the theoretical time complexity, but also resulted in practical and significant time savings. Without these optimizations, the time required to prove the type of statements appearing in Corollary \ref{cor_base_cases} (ii) and (iii) would likely increase asymptotically as $c_2 n^7$, as is suggested by Figure \ref{fig_computation}. 

\subsection{The proof}
Of foremost interest is whether we can prove Corollary \ref{cor_base_cases} with the algorithm. This turns out to be the case.

\begin{proof}[Proof of Corollary \ref{cor_base_cases}]
We computed the rank of $\mathbf{R}(Y)$ over the finite field $\Z_{8191}$ with prime characteristic $8191$. The optimized algorithm was able to compute the ranks of all of the matrices involved in Corollary \ref{cor_base_cases}, confirming in every instance that their rank coincided with the expected dimension of the subspace $W(n, s; a_1, a_2, a_3)$. We were particularly fortunate that in every case the maximal rank was found at the first random example that we sampled. The total execution time for proving Corollary \ref{cor_base_cases} with our experimental setup was $85.5$s for case (i), $8871.2$s for case (ii), $6917.7$s for case (iii), and $221.9$s for case (iv). The base cases could thus be proved in just under four and a half hours.

The output of the algorithm includes a certificate consisting of the explicit expressions of the linear forms $\ell_i, m_i \in \Z_{8191}^{n+1}$ and $\ell_{l,j}, m_{l,j} \in U_l$. At the corresponding points $p_i = [\ell_i^{d-1} m_i] \in \T_n$ and $q_{l,j} = [\ell_{l,j}^{d-1} m_{l,j}] \in \T_n \cap \Ss_l$, the linear span of the tangent spaces to the tangential variety $\T_n$ spans a subspace $W(n, s; a_1, a_2, a_3)$ of dimension exactly equal to the expected, i.e., maximal, value $w(n, s; a_1, a_2, a_3)$.
All the certificates produced by our algorithm may be found at the following web page: \url{https://doi.org/10.13140/RG.2.1.2843.3368}. They constitute the proof of Corollary \ref{cor_base_cases}.

To give an impression of the output of the program, an example is included below. The output of the algorithm for proving the truth of the only true equiabundant statement $T(7, 8; 0, 0, 0)$ is as follows.
\begin{verbatim}
Using random seed: 1440664437
l_0 = [6240 5559 4744 2128 3525 2499 7333 2585]
m_0 = [5179 5860 2731 4978 4356 4995  358 2752]
l_1 = [6524 4761 3599 7815 1716 2187 4195 7889]
m_1 = [1512 3708 6893 7109 5519 5965 5496 2212]
l_2 = [2484 8072 7956 3951 6365   63 6777   37]
m_2 = [5225 7196 2009 3291 6451 5475 2616 5079]
l_3 = [4096  596 3500 6582 5675 2959 6074 3891]
m_3 = [4798 7696  188 5184  578 1679 2657  335]
l_4 = [7882 7500 5717 2715 1488 1144 5362 5122]
m_4 = [3740 7615 3260 3859 2746   75 1181 1268]
l_5 = [5979  741 5874 6408 7902 5006 3801 6057]
m_5 = [5718 1256 7323 3359 1176 5753  675 3460]
l_6 = [4415 2885  403 5801  124 1935 8094 6722]
m_6 = [5366 1942 5568 1892 6945 5454 7057 5850]
l_7 = [4552 7106 6564 5562 6468 3805 3021 5507]
m_7 = [7463 2235 5324 6275 2378 2047 1639 7436]
Constructed the 120 x 128 matrix R(Y) in 0.081s.
Computed the rank of R(Y) over F_8191 in 0.007s.
Found 0 + 120 = 120 vs. 120 expected.
T(7, 8; 0, 0, 0) is TRUE (SUBABUNDANT)
Total computation took 0.088s.
\end{verbatim}
\end{proof}


\begin{thebibliography}{10}

\bibitem{Abo2010}
H.~Abo, \emph{On non-defectivity of certain {Segre--Veronese} varieties}, J.
  Symbolic Comput. \textbf{45} (2010), no.~12, 1254--1269.

\bibitem{abo2014}
\bysame, \emph{Varieties of completely decomposable forms and their secants},
  J.~Algebra \textbf{403} (2014), 135--153.

\bibitem{AB2009}
H.~Abo and M.~Brambilla, \emph{Secant varieties of {Segre--Veronese} varieties
  {$\P^m \times \P^n$} embedded by {$\mathcal{O}(1,2)$}}, Exp. Math.
  \textbf{18} (2009), no.~3, 369--384.

\bibitem{AB2013}
\bysame, \emph{On the dimensions of secant varieties of {Segre--Veronese}
  varieties}, Ann. Mat. Pura Appl. (4) \textbf{192} (2013), no.~1, 61--92.

\bibitem{AOP2009}
H.~Abo, G.~Ottaviani, and C.~Peterson, \emph{Induction for secant varieties of
  {S}egre varieties}, Trans. Amer. Math. Soc. \textbf{361} (2009), 767--792.

\bibitem{AH1995}
J.~Alexander and A.~Hirschowitz, \emph{Polynomial interpolation in several
  variables}, J. Algebraic Geometry \textbf{4} (1995), no.~2, 201--222.

\bibitem{AB2011}
E.~Arrondo and A.~Bernardi, \emph{On the variety parameterizing completely
  decomposable polynomials}, J. Pure Appl. Algebra \textbf{215} (2011),
  201--220.

\bibitem{ballico}
E.~Ballico, \emph{On the secant varieties to the tangent developable of a
  {Veronese} variety}, J. Algebra \textbf{288} (2005), no.~2, 279--286.

\bibitem{BCGI}
A.~Bernardi, M.~V. Catalisano, A.~Gimigliano, and M.~Id\'a, \emph{Secant
  varieties to osculating varieties of {Veronese} embeddings of {$\P^n$}}, J.
  Algebra \textbf{321} (2009), 982--1004.
  
\bibitem{BO2008}
MC. Brambilla and G.~Ottaviani, \emph{On the {A}lexander--{H}irschowitz
  theorem}, J. Pure Appl. Algebra \textbf{212} (2008), no.~5, 1229--1251.
 
\bibitem{Burgisser2000}
  P. B\"urgisser, \emph{Cook's versus Valiant's hypothesis}, Theor. Comput. Sci. 
  \textbf{235} (2000), no.~1, 71--88.

\bibitem{Burgisser1997}
 P. B\"urgisser, P. Clausen, and M.~A. Shokrollahi, \emph{Algebraic Complexity Theory}, 
  Grundlehren der mathematischen Wissenschaften 315, Springer Verlag-Berlin Heidelberg, 1997.
 

\bibitem{CGG2002}
M.~V. Catalisano, A.~V. Geramita, and A.~Gimigliano, \emph{On the secant
  varieties to the tangential varieties of a {Veronesean}}, Proc. Amer. Math.
  Soc. \textbf{130} (2002), no.~4, 975--985.

\bibitem{CGG2004}
\bysame, \emph{Higher secant varieties of {Segre--Veronese} varieties},
  Projective Varieties with Unexpected Properties: A Volume in Memory of
  Giuseppe Veronese (C.~Ciliberto, A.~V. Geramita, B.~Harbourne, Mir\'o-Roig
  R., and K.~Ranestad, eds.), Walter de Gruyter GmbH \& Co., KG, 10785 Berlin,
  Germany, 2004, pp.~81--107.

\bibitem{COV2014}
L.~Chiantini, G.~Ottaviani, and N.~Vannieuwenhoven, \emph{An algorithm for
  generic and low-rank specific identifiability of complex tensors}, SIAM J.
  Matrix Anal. Appl. \textbf{35} (2014), no.~4, 1265--1287.

\bibitem{COV2015}
\bysame, \emph{On generic identifiability of symmetric tensors of subgeneric
  rank}, Trans. Amer. Math. Soc. (2015), Accepted.

\bibitem{Davis}
T.~A. Davis, \emph{Direct methods for sparse linear systems}, Fundamentals of
  Algorithms, SIAM, Philadelphia, PA, USA, 2006.

\bibitem{Duff}
I.~S. Duff, A.~M. Erisman, and J.~K. Reid, \emph{Direct methods for sparse
  matrices}, Numerical Mathematics and Scientific Computation, Oxford
  University Press, New York, NY, USA, 1986.

\bibitem{FFLAS}
J.~Dumas, P.~Giorgi, and C.~Pernet, \emph{Dense linear algebra over word-size
  prime fields: the {FFLAS} and {FFPACK} packages}, ACM Trans. Math. Software
  \textbf{35} (2008), no.~3, 19:1--19:35.

\bibitem{GKZ1994}
I.~M. Gelfand, M.~M. Kapranov, and A.~V. Zelevinsky, \emph{Discriminants,
  resultants, and multidimensional determinants}, Mathematics: Theory \&
  Applications, Birkh\"auser Boston, 1994.

\bibitem{Eigen}
G.~Guennebaud, B.~Jacob, et~al., \emph{Eigen v3},
  \url{http://eigen.tuxfamily.org}, 2010.

\bibitem{LP2013}
A.~Laface and E.~Postinghel, \emph{Secant varieties of {Segre--Veronese}
  embeddings of {$(\P^1)^r$}}, Math. Ann. \textbf{356} (2013), no.~4,
  1455--1470.

\bibitem{Landsberg2012}
J.M. Landsberg, \emph{Tensors: Geometry and applications}, Graduate Studies in
  Mathematics, vol. 128, American Mathematical Society, Providence, Rhode
  Island, 2012.

\bibitem{Landsberg2015}
J.M. Landsberg, \emph{Geometric complexity theory: an introduction for geometers}, Ann. Math. Ferrara \textbf{61} (2015), 65--117.
  
\bibitem{OpenBLAS}
W.~Qian, Z.~Xianyi, Z.~Yunquan, and Q.~Yi, \emph{{AUGEM}: Automatically
  generate high performance dense linear algebra kernels on x86 {CPUs}},
  Proceedings of the International Conference on High Performance Computing,
  Networking, Storage and Analysis (New York, NY), ACM, 2013, pp.~25:1--25:12.

\bibitem{Shin}
Y.~Shin, \emph{Secants to the variety of completely reducible forms and the
  {Hilbert} function of the union of star-configurations}, J. Algebra Appl.
  \textbf{11} (2012), no.~6, 1--27.
  
\bibitem{Shpilka2010}
A. Shpilka and A. Yehudayoff, \emph{Arithmetic circuits: a survey of recent results and open questions},
  Found. Trends Theoretical Computer Science \textbf{5} (2010), no.~3-4, 207--388.

\bibitem{Terracini1911}
A.~Terracini, \emph{Sulla {$V_k$} per cui la variet\`a degli {$S_h$}
  $h+1$-secanti ha dimensione minore dell'ordinario}, Rend. Circ. Mat. Palermo
  \textbf{31} (1911), 392--396.

\bibitem{Torrance2015}
D.~A. Torrance, \emph{Generic forms of low {Chow} rank}, arXiv:1508.05546v1
  (2015), 1--11.

\end{thebibliography}


\end{document}